%% file: main.tex
\documentclass[11pt,reqno]{amsart}
\input{preamble.tex}
\title[\textbf{\lowercase{w}A2}-spaces, the Kastanas game and strategically Ramsey sets]{Weak \textbf{A2} spaces, the Kastanas game and strategically Ramsey sets}
\author{Clement Yung}

\begin{document}

\begin{abstract}
    We introduce the notion of a \emph{weak \textbf{A2} space} (or \textbf{wA2}-space), which generalises spaces satisfying Todor\v{c}evi\'{c}'s axioms \textbf{A1}-\textbf{A4} and countable vector spaces. We show that in any Polish weak \textbf{A2} space, analytic sets are Kastanas Ramsey, and discuss the relationship between Kastanas Ramsey sets and sets in the projective hierarchy. We also show that in all spaces satisfying \textbf{A1}-\textbf{A4}, every subset of $\cal{R}$ is Kastanas Ramsey iff Ramsey, generalising the recent result by \cite{CD23}. Finally, we show that in the setting of Gowers \textbf{wA2}-spaces, Kastanas Ramsey sets and strategically Ramsey sets coincide, providing a connection between the recent studies on topological Ramsey spaces and countable vector spaces.
\end{abstract}

\maketitle

\section{Introduction}
In this paper we show the notion of a \emph{weak \textbf{A2} space} provides a direct connection between the study of the abstract Kastanas game on closed triples $(\cal{R},\leq,r)$ satisfying Todor\v{c}evi\'{c}'s axioms \textbf{A1}-\textbf{A4} in \cite{CD23}, and the study of strategically Ramsey subsets of a countable vector space in \cite{R08}. We show that set-theoretic properties of strategically Ramsey subsets of countable vector spaces, and Ramsey subsets of topological Ramsey spaces, are consequences of the properties of Kastanas Ramsey sets in \textbf{wA2}-spaces. 

It was shown in \cite{CD23} that if $(\cal{R},\leq,r)$ is a closed triple $(\cal{R},\leq,r)$ satisfying Todor\v{c}evi\'{c}'s axioms \textbf{A1}-\textbf{A4}, and it is selective, then a subset of $\cal{R}$ is Kastanas Ramsey iff it is Ramsey. On the other hand, Rosendal showed that all analytic subsets of a countable vector space are strategically Ramsey, along with other set-theoretic behaviour of strategically Ramsey sets. This was made further abstract in \cite{d20}, where de Rancourt introduced the notion of a \emph{Gowers space}, and showed that all analytic subsets of a Gowers space are strategically Ramsey.

This paper presents three main theorems. Our first theorem generalises the result given in \cite{CD23} to all spaces satisfying \textbf{A1}-\textbf{A4}.

\begin{theorem}
\label{thm:kr.iff.r}
    Suppose that $(\cal{R},r,\leq)$ is a closed triple satisfying \textbf{A1}-\textbf{A4}. Then a set $\X \subseteq \cal{R}$ is Kastanas Ramsey iff it is Ramsey.
\end{theorem}

Observing that the abstract Kastanas game may be similarly studied on \textbf{wA2}-spaces, we present a few set-theoretic properties of the set of Kastanas Ramsey subsets of a \textbf{wA2}-space $\cal{R}$. If $\cal{AR}$ is countable, then $\cal{R}$ is a Polish space under the usual metrisable topology, so we may consider the projective hierarchy of subsets of $\cal{R}$. 

\begin{theorem}
\label{thm:analytic.is.kr}
    Suppose that $(\cal{R},\leq,r)$ is a \textbf{wA2}-space, and that $\cal{AR}$ is countable (so $\cal{R}$ is a Polish space under the metrisable topology). Then every analytic set is Kastanas Ramsey. 
\end{theorem}

We will also show that for a large class of \textbf{wA2}-spaces, Theorem \ref{thm:analytic.is.kr} is consistently optimal. See Theorem \ref{thm:sigma_1^2.wo.gives.non.kr.set} and Corollary \ref{cor:r.times.2^omega.has.coanalytic.non.kr.set}.

In our last section, we study the corresponding version of the Kastanas game on Gowers space\footnote{de Rancourt also introduced the Kastanas game $K_p$ in \cite{d20}, which differs from the one this paper shall be presenting.}. We relate the two concepts by introducing the notion of a \emph{Gowers \textbf{wA2}-space} (in which countable vector spaces are an example of).

\begin{theorem}
    Let $(\cal{R},r,\leq)$ be a Gowers \textbf{wA2}-space. Then the Kastanas game on $\cal{R}$ (as a Gowers space) and the Kastanas game on $\cal{R}$ (as a \textbf{wA2}-space) are equivalent. Furthermore, a subset of $\cal{R}$ is Kastanas Ramsey iff it is strategically Ramsey.
\end{theorem}

For the precise statements, see Theorem \ref{thm:kr.strategies.equivalence} and Proposition \ref{prop:kr.iff.sr}.

\section{Weak \textbf{A2} spaces}
In this section, we provide a recap of the axioms of topological Ramsey spaces presented by Todor\v{c}evi\'{c} in \cite{T10}, which would preface the setup of a \textbf{wA2}-space. We follow up by discussing various examples of \textbf{wA2}-spaces, and an overview of Ramsey subsets of \textbf{wA2}-spaces, motivating the need to study an alternative variant of an infinite-dimensional Ramsey property.

\subsection{Axioms} 
We recap the four axioms presented by Todor\v{c}evi\'{c} in \cite{T10}, which are sufficient conditions for a triple $(\cal{R},\leq,r)$ to be a topological Ramsey space. Here, $\cal{R}$ is a non-empty set, $\leq$ be a quasi-order on $\cal{R}$, and $r : \cal{R} \times \omega \to \cal{AR}$ is a function. We also define a sequence of maps $r_n : \cal{R} \to \cal{AR}$ by $r_n(A) := r(A,n)$ for all $A \in \cal{R}$. Let $\cal{AR}_n \subseteq \cal{AR}$ be the image of $r_n$ (i.e. $a \in \cal{AR}_n$ iff $a = r_n(A)$ for some $A \in \cal{R}$).

The four axioms are as follows:
\begin{enumerate}
    \item[\textbf{(A1)}] 
    \begin{enumerate}[label=(\arabic*)]
        \item $r_0(A) = \emptyset$ for all $A \in \cal{AR}$.

        \item $A \neq B$ implies $r_n(A) \neq r_n(B)$ for some $n$.

        \item $r_n(A) = r_m(B)$ implies $n = m$ and $r_k(A) = r_k(B)$ for all $k < n$.
    \end{enumerate}
    For each $a \in \cal{AR}$, let $\lh(a)$ denote the unique $n$ in which $a \in \cal{AR}_n$. By Axiom \textbf{A1}(3), this $n$ is well-defined.

    \item[\textbf{(A2)}]  There is a quasi-ordering $\leq_\fin$ on $\cal{AR}$ such that:
    \begin{enumerate}[label=(\arabic*)]
        \item $\{a \in \cal{AR} : a \leq_\fin b\}$ is finite for all $b \in \cal{AR}$.

        \item $A \leq B$ iff $\forall n \exists m [r_n(A) \leq_\fin r_m(B)]$.

        \item $\forall a,b \in \cal{AR}[a \sqsubseteq b \wedge b \leq_\fin c \to \exists d \sqsubseteq c[a \leq_\fin d]]$.
    \end{enumerate}

    \item[\textbf{(A3)}] We may define the \emph{Ellentuck neighbourhoods} as follows: For any $A \in \cal{R}$, $a \in \cal{AR}$ and $n \in \N$, we let:
    \begin{align*}
        [a,A] &:= \{B \in \cal{R} : B \leq A \wedge \exists n[r_n(B) = a]\}, \\
        [n,A] &:= [r_n(A),A].
    \end{align*}
    Then the depth function defined by, for $B \in \cal{R}$ and $a \in \cal{AR}$:
    \begin{align*}
        \depth_B(a) := 
        \begin{cases}
            \min\{n < \omega : a \leq_\fin r_n(B)\}, &\text{if such $n$ exists}, \\
            \infty, &\text{otherwise}.
        \end{cases}
    \end{align*}
    satisfies the following:
    \begin{enumerate}[label=(\arabic*)]
        \item If $\depth_B(a) < \infty$, then for all $A \in [\depth_B(a),B]$, $[a,A] \neq \emptyset$.

        \item If $A \leq B$ and $[a,A] \neq \emptyset$, then there exists $A' \in [\depth_B(a),B]$ such that $\emptyset \neq [a,A'] \subseteq [a,A]$.
    \end{enumerate}
    For each $A \in \cal{R}$, we let:
    \begin{align*}
        \cal{AR}\restrictedto A := \{a \in \cal{AR} : \exists n[a \leq_\fin r_n(A)]\}.
    \end{align*}
    If $a \in \cal{AR}\restrictedto A$, we also define:
    \begin{align*}
        \cal{AR}\restrictedto[a,A] &:= \{b \in \cal{AR}\restrictedto A : a \sqsubseteq b\}, \\
        r_n[a,A] &:= \{b \in \cal{AR}\restrictedto[a,A] : \lh(b) = n\}.
    \end{align*}

    \item[\textbf{(A4)}] If $\depth_B(a) < \infty$ and if $\O \subseteq \cal{AR}_{\lh(a)+1}$, then there exists $A \in [\depth_B(a),B]$ such that $r_{\lh(a)+1}[a,A] \subseteq \O$ or $r_{\lh(a)+1}[a,A] \subseteq \O^c$.
\end{enumerate}

We introduce a useful weakening of Axiom \textbf{A2}, which we shall call \emph{weak} \textbf{A2}, or just \textbf{wA2}.

\begin{axiom}[\textbf{wA2}]
    There is a quasi-ordering $\leq_\fin$ on $\cal{AR}$ such that:
    \begin{enumerate}
        \item[(w1)] $\{a \in \cal{AR} : a \leq_\fin b\}$ is countable for all $b \in \cal{AR}$.
    
        \item[(2)] $A \leq B$ iff $\forall n \exists m [r_n(A) \leq_\fin r_m(B)]$.
    
        \item[(3)] $\forall a,b \in \cal{AR}[a \sqsubseteq b \wedge b \leq_\fin c \to \exists d \sqsubseteq c[a \leq_\fin d]]$.
    \end{enumerate}
\end{axiom}

Note that by axiom \textbf{A1}, we may identify each element $A \in \cal{R}$ as a sequence of elements of $\cal{AR}$, via the map $A \mapsto (r_n(A))_{n<\omega}$. Therefore, we may identify $\cal{R}$ as a subset of $\cal{AR}^\N$.

\begin{definition}
    A triple $(\cal{R},\leq,r)$ is said to be:
    \begin{enumerate}
        \item a \emph{closed triple} if $\cal{R}$ is a metrically closed subset of $\cal{AR}^\N$.
        
        \item a \emph{\textbf{wA2}-space} if $(\cal{R},\leq,r)$ is a closed triple satisfying Axioms \textbf{A1}, \textbf{wA2} and \textbf{A3}.

        \item an \emph{\textbf{A2}-space} if it is a \textbf{wA2}-space satisfying \textbf{A2}.
    \end{enumerate}
\end{definition}

Given a \textbf{wA2}-space $(\cal{R},\leq,r)$, we shall forcus on the following two topologies on $\cal{R}$.
\begin{enumerate}
    \item The metrisable topology - we may equip $\cal{R}$ with the first difference metric, where for $A,B \in \cal{R}$, $d(A,B) = \frac{1}{2^n}$ where $n$ is the least integer such that $r_n(A) \neq r_n(B)$. If $\cal{AR}$ is countable, then under this metrisable topology, $\cal{R}$ is a Polish space.
    
    \item The Ellentuck topology generated by open sets of the form $[a,A]$ for $A \in \cal{R}$ and $a \in \cal{AR}\restrictedto A$.
\end{enumerate}

In this paper, unless stated otherwise all topological properties of $\cal{R}$ will be with respect to the metrisable topology.

\subsection{Examples}
We discuss several examples of \textbf{wA2}-spaces. All examples below, except for countable vector spaces and the singleton space (Example \ref{ex:vector.spaces} and \ref{ex:singleton.space}), may be found in \cite{T10}. Discussions of countable vector spaces and strategically Ramsey sets may be found in \cite{R08}, \cite{R09}, \cite{S18} and \cite{d20}. 

\begin{example}[Natural numbers/Ellentuck space $\lbrack \N\rbrack^\infty$]
\label{ex:natural.numbers}
    Let $\cal{R} = [\N]^\infty$. For each $n$ and $A \in [\N]^\infty$, let $r_n(A)$ be the finite set containing the $n$ least elements of $A$ (so $\cal{AR} = [\N]^{<\infty}$). Let $\leq_\fin$ denote the subset relation $\subseteq$. It is easy to check that $([\N]^\infty,\subseteq,r)$ is a closed triple satisfying \textbf{A1}-\textbf{A4}.
\end{example}

\begin{example}[Infinite block sequences/Gowers' space $\FIN_k^{[\infty]}$]
\label{ex:gowers.FIN_k}
    For each $k < \omega$, let $\FIN_k$ be the set of all functions $x : \N \to \N$ with finite support (i.e. the set $\supp(x) := \{n \in \N : x(n) > 0\}$ is finite), such that $\ran(x) \subseteq \{0,\dots,k\}$ and $k \in \ran(x)$. For $x,y \in \FIN_k$, we also define the following:
    \begin{enumerate}
        \item (Tetris operation) $T(x)(n) := \max\{x(n) - 1,0\} \in \FIN_{k-1}$.

        \item $x < y \iff \max(\supp(x)) < \min(\supp(y))$.

        \item If $x < y$, define $(x + y)(n) := \max\{x(n),y(n)\} \in \FIN_k$.
    \end{enumerate}
    It is easy to see that $<$ is transitive. We let $\cal{R} = \FIN_k^{[\infty]}$ be the set of all infinite $<$-increasing sequences, and for each $A = (x_n)_{n<\omega} \in \cal{R}$, $r_N(A) := (x_n)_{n<N}$. We have that $\cal{AR} = \FIN_k^{[<\infty]}$, the set of all finite $<$-increasing sequences.
    
    Given $a = (x_n)_{n<N} \in \FIN_k^{[<\infty]}$, we let:
    \begin{align*}
        \c{a} := \{T^{\lambda_0}(x_0) + \cdots + T^{\lambda_{N-1}}(x_{N-1}) : \lambda_i = 0 \text{ for some } i < N\}.
    \end{align*}
    For two $a,b \in \FIN_k^{[<\infty]}$, we write $a \leq_\fin b$ iff $\c{a} \subseteq \c{b}$. Then $(\FIN_k^{[\infty]},\leq,r)$ is a closed triple satisfying \textbf{A1}-\textbf{A4}. 
    
    $\FIN_k^{[\infty]}$ was first introduced by Gowers in \cite{G92} and \cite{G02} to resolve several long-standing problems in Banach space theory. The current formulation of $\FIN_k^{[\infty]}$ may be found in \cite{T10}. The fact that $(\FIN_k^{[\infty]},\leq,r)$ satisfies \textbf{A4} follows from Gowers' $\FIN_k$ theorem (Theorem 1 \cite{G92}, or Theorem 2.2, \cite{T10}).
\end{example}

\begin{example}[Infinite block sequences $\FIN_{\pm k}^{[\infty]}$]
\label{ex:gowers.FIN_pmk}
    Consider a setup simi;ar to Example \ref{ex:gowers.FIN_k}. Instead, we let $\FIN_{\pm k}$ be the set of all functions $x : \N \to \Z$ with finite support, such that $\ran(x) \subseteq \{-k,\dots,k\}$, and ($k \in \ran(x)$ or $-k \in \ran(x)$). The tetris operation is now modified to:
    \begin{align*}
        T(x)(n) := 
        \begin{cases}
            x(n) - 1, &\text{if $x(n) > 0$}, \\
            x(n) + 1, &\text{if $x(n) < 0$}, \\
            0, &\text{otherwise}.
        \end{cases}
    \end{align*}
    The rest of the setup is similar. We let $\cal{R} = \FIN_{\pm k}^{[\infty]}$ be the set of all infinite $<$-increasing sequences, so $\cal{AR} = \FIN_{\pm k}^{[<\infty]}$ is the set of all finite $<$-increasing sequences.
    
    Given $a = (x_n)_{n<N} \in \FIN_{\pm k}^{[<\infty]}$, we let:
    \begin{align*}
        \c{a} := \{\varepsilon_0 T^{\lambda_0}(x_0) &+ \cdots + \varepsilon_{N-1} T^{\lambda_{N-1}}(x_{N-1}) : \\
        &\varepsilon_0,\dots,\varepsilon_{N-1} \in \{\pm 1\} \text{ and } \lambda_i = 0 \text{ for some } i < N\}.
    \end{align*}
    Then $\FIN_{\pm k}^{[\infty]}$ is an \textbf{A2}-space. However, $\FIN_{\pm k}^{[\infty]}$ does not satisfy \textbf{A4} - for each $x \in \FIN_{\pm k}$, let:
    \begin{align*}
        n_x := \min\{n < \omega : x(n) \in \{\pm k\}\}.
    \end{align*}
    Now let:
    \begin{align*}
        Y := \{x \in \FIN_{\pm k} : x(n_x) = k\}.
    \end{align*}
    Then, for all $A \in \FIN_{\pm k}^{[\infty]}$, $\c{A} \cap Y \neq \emptyset$ and $\c{A} \cap Y^c \neq \emptyset$.

    $\FIN_{\pm k}^{[\infty]}$ was first introduced by Gowers in \cite{G92}, and the current formulation of $\FIN_{\pm k}^{[\infty]}$ may be found in \cite{T10}.
\end{example}

\begin{example}[Hales-Jewett space $W_{Lv}^{[\infty]}$]
    Let $L = \bigcup_{n=0}^\infty L_n$ be a countable increasing union of finite alphabets with variable $v \notin L$. Let $W_{Lv}$ denote the set of all \emph{variable-words} over $L \cup \{v\}$, i.e. all finite strings of elements of $L \cup \{v\}$ in which $v$ occurs at least once. For each $x \in W_{Lv}$ and $\lambda \in L \cup \{v\}$, we let $x[\lambda]$ denote the word in which all $v$'s occurring in $x$ are replaced with $a$.
    
    Given $x_0,\dots,x_n \in W_{Lv}$, we write $(x_i)_{i<n} < x_n$ iff $\sum_{i<n} |x_i| < |x_n|$. A sequence $(x_n)_{n<N}$ is \emph{rapidly increasing} if $(x_i)_{i<n} < x_n$ for all $n < N$. Let $\cal{R} = W_{Lv}^{[\infty]}$ be the set of all infinite rapidly increasing sequences, and for each $A = (x_n)_{n<\omega} \in \cal{R}$, $r_N(A) := (x_n)_{n<N}$. We have that $\cal{AR} = W_{Lv}^{[<\infty]}$, the set of all finite rapidly increasing sequences.
    
    Given $a = (x_n)_{n<N} \in W_{Lv}^{[<\infty]}$, we let:
    \begin{align*}
        \c{a} := \{x_{m_0}[\lambda_0]^\frown\cdots^\frown x_{\lambda_k}[\lambda_k] : m_0 < \cdots m_k \text{ and } \lambda_i = v \text{ for some } i\}.
    \end{align*}
    For two $a,b \in W_{Lv}^{[<\infty]}$, we write $a \leq_\fin b$ iff $\c{a} \subseteq \c{b}$, and $\c{a} \not\subseteq \c{c}$ for all $c \sqsubseteq b$ and $c \neq b$. Then $(W_{Lv}^{[\infty]},\leq,r)$ is a closed triple satisfying \textbf{A1}-\textbf{A4}.
    
    $W_{Lv}^{[\infty]}$ was first introduced by Hales-Jewett in \cite{HJ63}, and the current formulation of $W_{Lv}^{[\infty]}$ may be found in \cite{T10}.
\end{example}

\begin{example}[Strong subtrees $\S_\infty$]
\label{ex:strong.subtrees}
    Let $T \subseteq \omega^{<\omega}$ be a tree (i.e. a subset that is closed under initial segments). We introduce some terminologies.
    \begin{enumerate}
        \item A node $s \in T$ \emph{splits} in $T$ if there exist $m \neq n$ such that $s^\frown m \in T$ and $s^\frown n \in T$. For $n \geq 1$, we let $\func{split}_n(T)$ be the set of all $s \in T$ such that $s$ splits in $T$, and there are $k_1 < \cdots < k_{n-1} < \dom(s)$ such that $s\restrictedto k_i$ splits for $1 \leq i < n$. We then let $\func{split}(T) := \bigcup_{n<\omega}\func{split}_n(T)$.

        \item A node $s \in T$ is \emph{terminal} in $T$ if $s^\frown n \notin T$ for all $n$.

        \item The \emph{height} of a non-empty tree $T$ is defined by:
        \begin{align*}
            \height(T) := \max(\{0\} \cup \{n \geq 1 : \func{split}_n(T) \neq \emptyset\}).
        \end{align*}

        \item $T$ is \emph{perfect} if for all $s \in T$, either there exists some $t \sqsupseteq s$ such that $t$ splits in $T$, or $s$ is terminal in $T$.
    \end{enumerate}
    Let $\cal{R} = \S_\infty$ be the set of \emph{strong subtrees} $A \subseteq \omega^{<\omega}$. That is, $T$ is a perfect subtree which satisfies the following conditions:
    \begin{enumerate}
        \item For all $s,t \in A$ such that $\dom(s) = \dom(t)$, $s$ splits in $T$ iff $t$ splits in $A$.

        \item If $s \in A$ and there exists some $t \in A$ such that $\dom(s) < \dom(t)$, then there exists some $u \in A$ such that $s \sqsubseteq u$ and $\dom(u) = \dom(t)$.
    \end{enumerate}
    For each $A \in \S_\infty$ and $n < \omega$, we define:
    \begin{align*}
        r_n(A) := \{s \in A : |\{t \in A : t \sqsubseteq s \wedge t \text{ splits}\}| \leq n\}.
    \end{align*}
    We observe that if $A \in \S_\infty$, then $r_n(A)$ is a strong subtree of height $n$, so $\cal{AR} = \S_{<\infty}$ is the set of all strong subtrees of finite height. For $a,b \in \S_{<\infty}$, we write $a \leq_\fin b$ iff $a \subseteq b$, and for all nodes $s \in a$ which are terminal in $a$, $s$ splits in $b$ or is terminal $b$. Then $(\S_\infty,\leq,r)$ is a closed triple satisfying \textbf{A1}-\textbf{A4}.
    
    The space $\S_\infty$ is also known as the Milliken space of strong subtrees, and was first introduced by Milliken in \cite{M81} to generalise Silver's partition theorem for infinite trees. The theorem asserting that $(\S_\infty,\leq,r)$ satisfies \textbf{A4} (also proven in \cite{M81}) is the strong subtree variant of the Halpern-L\"{a}uchli theorem. The current formulation may be found in \cite{T10}.
\end{example}

\begin{example}[Carlson-Simpson Space $\E^\infty$]
\label{ex:carlson-simpson}
    Let $\cal{R} = \E_\infty$ denote the Carlson-Simpson space, i.e. the collection of equivalence relations $A$ on $\mathbb{N}$ in which $\mathbb{N}/A$ is infinite. For each $x \in \N$, let $[x]_A$ be the equivalence class of $A$ containing $x$. We then let $p(A) = \{p_n(A) : n < \omega\}$ be the increasing enumeration of the set of all minimal representatives of the equivalence classes of $A$. Note that $p_0(A) = 0$ always. For each $A \in \E_\infty$, we let $r_n(A) := A\restrictedto p_n(A)$, i.e. the restriction of the equivalence relation $A$ to the \emph{domain} $\dom(A\restrictedto p_n(A)) := \{0,1,\dots,p_n(A) - 1\}$. We denote $\cal{AE}_\infty := \cal{AR}$. For $a,b \in \cal{AE}_\infty$, we write $a \leq_\fin b$ iff $\dom(a) = \dom(b)$ and $a$ is coarser than $b$. We remark that for all $a \in \cal{AE}_\infty$, $\lh(a)$ is the number of equivalence classes in $a$. Then $(\E_\infty,\leq,r)$ is a closed triple satisfying \textbf{A1}-\textbf{A4}. 
    
    $\E^\infty$ was first introduced by Carlson-Simpson in \cite{CS90}, as part of their development of topological Ramsey theory. The current formulation may be found in \cite{T10}.
\end{example}

\begin{example}[Countable vector spaces $E^{[\infty]}$]
\label{ex:vector.spaces}
    Let $\mathbb{F}$ be a countable field, and let $E$ be an $\mathbb{F}$-vector space of dimension $\aleph_0$, with a distinguished Hamel basis $(e_n)_{n<\omega}$. Given $x \in E$, if $x = \sum_{n<\omega} a_ne_n$, we write $\supp(x) := \{n < \omega : a_n \neq 0\}$. We then define a partial order $<$ on $E \setminus \{0\}$ by:
    \begin{align*}
        x < y \iff \max(\supp(x)) < \min(\supp(y)).
    \end{align*}
    We let $\cal{R} := E^{[\infty]}$ denote the set of all infinite $<$-increasing sequences of non-zero vectors in $E$, and for each $A = (x_n)_{n<\omega} \in E^{[\infty]}$, define $r_N(A) := (x_n)_{n<N}$. We have that $\cal{AR} = E^{[<\infty]}$ is the set of finite $<$-increasing sequences of non-zero vectors. For two $a := (x_n)_{n<N}, b := (y_m)_{m<M} \in E^{[<\infty]}$, we write $(x_n)_{n<N} \leq_\fin (y_m)_{m<M}$ iff $x_n \in \func{span}\{y_m : m < M\}$ for all $n < N$. Then $(E^{[\infty]},\leq,r)$ is a \textbf{wA2}-space. Furthermore:
    \begin{enumerate}
        \item $(E^{[\infty]},\leq,r)$ is an \textbf{A2}-space iff $\mathbb{F}$ is finite - if $\mathbb{F}$ is finite, then for all $a = (x_n)_{n<N},b = (y_m)_{m<M} \in E^{[<\infty]}$, $a \leq_\fin b$ iff $x_n$ is an $\mathbb{F}$-linear combination of $y_0,\dots,y_{M-1}$, of which there are only finitely many of. If $\mathbb{F}$ is infinite, then $(e_0 + \lambda e_1) \leq_\fin (e_0,e_1)$ for all $\lambda \in \mathbb{F}$, so \textbf{A2} fails.
        
        \item $(E^{[\infty]},\leq,r)$ satisfies \textbf{A4} iff $|\mathbb{F}| = 2$. If $|\mathbb{F}| = 2$, then \textbf{A4} follows from Hindman's theorem (Theorem 2.41, \cite{T10}). If $|\mathbb{F}| > 2$, then define the set:
        \begin{align*}
            Y := \{x \in E : \text{$x = e_n + y$ for some $n$ and $e_n < y$}\}.
        \end{align*}
        It is not difficult to show that $\c{A} \cap Y \neq \emptyset$ and $\c{A} \cap Y^c \neq \emptyset$ for all $A \in E^{[<\infty]}$.
    \end{enumerate} 
    This Ramsey-theoretic framework of countable vector spaces were first introduced by Rosendal in \cite{R08} and \cite{R09}. He studied the set-theoretic properties of strategically Ramsey sets in this framework, a notion motivated by the Ramsey-theoretic methods employed by Gowers in \cite{G02}. Smythe studied the local Ramsey theory of this framework in \cite{S18}, extending some results by Rosendal to $\H$-strategically Ramsey sets, where $\H$ is a \emph{family} satisfying some combinatorial properties.
\end{example}

\begin{example}[Singleton space]
\label{ex:singleton.space}
    Let $\cal{R} = \{(0,0,\dots)\}$, the singleton containing the zero sequence. We define $r_n(A) := (0,\dots,0)$ of length $n$, and $\leq_\fin$ to be the equality relation. Then $(\cal{R},\leq,r)$ is a closed triple satisfying \textbf{A1}-\textbf{A4}. Then $(\cal{R},\leq,r)$ is a closed triple satisfying \textbf{A1}-\textbf{A4}. The singleton space serves as a pathological example of a topological Ramsey space.
\end{example}

\subsection{Ramsey sets}
The definition of a Ramsey subset of a topological Ramsey space may be extended to any \textbf{wA2}-spaces.

\begin{definition}
    Let $(\cal{R},\leq,r)$ be a \textbf{wA2}-space. A set $\X \subseteq \cal{R}$ is \emph{Ramsey} if for all $A \in \cal{R}$ and $a \in \cal{AR}\restrictedto A$, there exists some $B \in [a,A]$ such that $[a,B] \subseteq \X$ or $[a,B] \subseteq \X^c$.
\end{definition}

By the abstract Ellentuck theorem (Theorem 5.4, \cite{T10}), if $(\cal{R},\leq,r)$ is a closed triple satisfying \textbf{A1}-\textbf{A4}, a subset of $\cal{R}$ is Ramsey iff it is Baire relative to the Ellentuck topology. Since the Ellentuck topology refines the metrisable topology, every subset of $\cal{R}$ which is Baire relative to the metrisable topology is Ramsey. We show that \textbf{A4} is a necessary condition.

\begin{proposition}
\label{prop:a4.iff.ellentuck.theorem}
    Let $(\cal{R},\leq,r)$ be an \textbf{A2}-space. The following are equivalent:
    \begin{enumerate}
        \item $(\cal{R},\leq,r)$ satisfies \textbf{A4}.
        
        \item Every clopen subset of $\cal{R}$ is Ramsey,
    \end{enumerate}
\end{proposition}

\begin{proof}
    (1)$\implies$(2) follows from the abstract Ellentuck theorem. For the converse, let $A \in \cal{R}$, $a \in \cal{AR}\restrictedto A$ and $\O \subseteq \cal{AR}_{\lh(a)+1}$. Define:
    \begin{align*}
        \X := \{C \in \cal{R} : r_{\lh(a)+1}(C) \in \O\}.
    \end{align*}
    Since $\X$ is clopen, it is Ramsey. Therefore, there exists some $B \in [a,A]$ such that $[a,B] \subseteq \X$ or $[a,B] \subseteq \X^c$. If $[a,B] \subseteq \X$, then $r_{\lh(a)+1}(B) \subseteq \O$, and if $[a,B] \subseteq \X^c$, then $r_{\lh(a)+1}(B) \subseteq \O^c$. By \textbf{A3}, we may let $B' \in [\depth_A(a),A]$ such that $[a,B'] \subseteq [a,B]$, so $B'$ witnesses that \textbf{A4} holds for $\O$.
\end{proof}

Furthermore, Ramsey subsets of a \textbf{wA2}-space need not be closed under countable intersections (and are hence not closed under countable unions). 

\begin{example}[Countable vector space $E^{[\infty]}$]
    Let $\mathbb{F}$ be a countable field such that $|\mathbb{F}| > 2$, and let $E$ be an $\mathbb{F}$-vector space of dimension $\aleph_0$. Let $Y \subseteq E$ be the set defined in Example \ref{ex:vector.spaces} such that for all $A \in E^{[\infty]}$, $\c{A} \cap Y \neq \emptyset$ and $\c{A} \cap Y^c \neq \emptyset$. We define $Y_n \subseteq E$ for each $n$ such that $Y_n^c$ is finite for all $n$, and $Y = \bigcap_{n<\omega} Y_n$. This is possible as $E$ is countable.

    For each $n$, we let:
    \begin{align*}
        \X_n &:= \{(x_n)_{n<\omega} \in E^{[\infty]} : x_0 \in Y_n\}, \\
        \X &:= \{(x_n)_{n<\omega} \in E^{[\infty]} : x_0 \in Y\}.
    \end{align*}
    Note that $\X = \bigcap_{n<\omega} \X_n$. 
    
    \begin{claim}
        $\X_n$ is Ramsey for all $n$.
    \end{claim}

    \begin{midproof}
        Let $A \in E^{[\infty]}$ and $a \in E^{[<\infty]}\restrictedto A$. If $a = (x_0,\dots,x_{n-1}) \neq \emptyset$, then $[a,A] \subseteq \X_n$ or $\X_n^c$, depending if $x_0 \in Y_n$ or $x_0 \in Y_n^c$. Otherwise, let $B \leq A$ be such that $x < B$ for all $x \in Y_n^c$, which is possible as $Y_n^c$ is finite. Then $[\emptyset,B] \subseteq \X_n^c$.
    \end{midproof}

    However, $\X$ is not Ramsey - for all $A \in E^{[\infty]}$, there exist $(x_n)_{n<\omega} \leq A$ and $(y_n)_{n<\omega} \leq A$ such that $x_0 \in Y$ and $y_0 \in Y^c$, so $[\emptyset,A] \not\subseteq \X$ and $[\emptyset,A] \not\subseteq \X^c$.
\end{example}

These observations show that Ramsey sets in \textbf{wA2}-spaces are not as well-behaved as Ramsey sets in topological Ramsey spaces, prompting us to consider an alternative notion of Ramsey sets in \textbf{wA2}-spaces - one example being the notion of Kastanas Ramsey. 

\section{The Kastanas game in \textbf{wA2}-spaces}
We shall introduce the abstract Kastanas game in \textbf{wA2}-spaces, and study the set-theoretic properties of Kastanas Ramsey sets. 

\subsection{The abstract Kastanas game}
\begin{definition}[Definition 5.1, \cite{CD23}]
\label{def:kastanas.game}
    Let $(\cal{R},\leq,r)$ be a \textbf{wA2}-space. Let $A \in \cal{R}$ and $a \in \cal{AR}\restrictedto A$. The \emph{Kastanas game} played below $[a,A]$, denoted as $K[a,A]$, is defined as a game played by Player \textbf{I} and \textbf{II} in the following form:
    \begin{center}
        \begin{tabular}{c|c|c}
            Turn & \textbf{I} & \textbf{II} \\
            \hline
            1 & $A_0 \in [a,A]$ & \\
            & & $a_1 \in r_{\lh(a)+1}[a,A_0]$ \\
            & & $B_0 \in [a_1,A_0]$ \\
            \hline
            2 & $A_1 \in [a_1,B_0]$ & \\
            & & $a_2 \in r_{\lh(a_1)+1}[a_1,A_1]$ \\
            & & $B_1 \in [a_2,A_1]$ \\
            \hline
            3 &  $A_2 \in [a_2,B_1]$ & \\
            & & $a_3 \in r_{\lh(a_2)+1}[a_2,A_2]$ \\
            & & $B_2 \in [a_3,A_2]$ \\
            \hline
            $\vdots$ & $\vdots$ & $\vdots$
        \end{tabular}
    \end{center}
    The outcome of this game is $\lim_{n \to \infty} a_n \in \cal{R}$ (i.e. the unique element $B \in \cal{R}$ such that $r_n(B) = a_n$ for all $n$). We say that \textbf{I} (resp. \textbf{II}) has a strategy in $K[a,A]$ to \emph{reach} $\X \subseteq \cal{R}$ if it has a strategy in $K[a,A]$ to ensure that the outcome is in $\X$. 
\end{definition}

Note that we do not require $(\cal{R},\leq,r)$ to satisfy either \textbf{A2} or \textbf{A4} for the game to make sense. In particular, we may consider the abstract Kastanas game in countable vector spaces.

\begin{definition}
\label{def:kastanas.ramsey}
    Let $(\cal{R},\leq,r)$ be a \textbf{wA2}-space. A set $\X \subseteq \cal{R}$ is \emph{Kastanas Ramsey} if for all $A \in \cal{R}$ and $a \in \cal{AR}\restrictedto A$, there exists some $B \in [a,A]$ such that one of the following holds:
    \begin{enumerate}
        \item \textbf{I} has a strategy in $K[a,B]$ to reach $\X^c$.

        \item \textbf{II} has a strategy in $K[a,B]$ to reach $\X$.
    \end{enumerate}
\end{definition}

The seemingly unintuitive decision to define Kastanas Ramsey sets such that \textbf{I} plays into the set $\X^c$, instead of $\X$, allows us to describe the relationship between Kastanas Ramsey sets and strategically Ramsey sets, projections and projective sets more easily.

We conclude the section with some definitions and notations which are useful in studying the Kastanas game. 

\begin{definition}
\label{def:states}
    Let $(\cal{R},\leq,r)$ be a \textbf{wA2}-space, and let $A \in \cal{R}$ and $a \in \cal{AR}\restrictedto A$. Consider the Kastanas game $K[a,A]$.
    \begin{enumerate}
        \item A \emph{(partial) state} is a tuple containing the plays made by both players in a partial play of the game $K[a,A]$. For instance, a state ending on the $n^\text{th}$ turn of \textbf{I} would be:
        \begin{align*}
            s = (A_0,a_1,B_0,A_1,\dots,A_{n-1})
        \end{align*}
        The \emph{rank} of $s$ would be the turn number in which the last play was made (so, in the example above, $\rank(s) = n$).

        A state \emph{for \textbf{I}} (resp. \emph{for \textbf{II}}) is a state as defined above, except only the plays made by \textbf{I} (resp. by \textbf{II}) are listed in the tuple. For instance a state for \textbf{I} would be:
        \begin{align*}
            s_\textbf{I} = (A_0,A_1,\dots,A_{n-1}),
        \end{align*}
        and a state for \textbf{II} would be:
        \begin{align*}
            s_\textbf{II} = (a_1,B_0,\dots,a_n,B_{n-1}).
        \end{align*}

        \item If $s = (A_0,a_1,B_0,\dots,A_n)$ is a state of rank $n$, then the \emph{realisation} of $s$, denoted as $a(s)$, is the element of $\cal{AR}$ last played by \textbf{II}, i.e. $a_{n-1}$. We also say that $s$ \emph{realises} $a(s)$. If $\rank(s) = 0$, then $a(s) := a$. 

        If $\sigma$ is a strategy for \textbf{I} (resp. \textbf{II}) in $K[a,A]$ and $s$ is state for \textbf{I} (resp. \textbf{II}) following $\sigma$, then $a(s)$ is understood to mean the element $a(s')$ (i.e. $a_n$), where $s'$ is the state, following $\sigma$, such that $s$ is the play made by \textbf{I} (resp. \textbf{II}) in $s'$.

        \item A \emph{total state} $s$ is an infinite sequence of plays made by both players in a total play of the game $K[a,A]$. Thus, a total state $s$ would be of the form:
        \begin{align*}
            s = (A_0,a_1,B_0,A_1,a_2,B_1,\dots).
        \end{align*}
        The \emph{realisation} of $s$ would be the element $A(s) := \lim_{n \to \infty} a_n$, i.e. the unique element $A(s) \in \cal{R}$ such that $r_{\lh(a_0)+n}(A) = a_n$ for all $n$.

        \item If $s$ is a state (for \textbf{I} or \textbf{II} resp.), and $n$ is such that either $n \leq \rank(s)$ or $s$ is a total state, then write the \emph{restriction} of $s$ to rank $n$, denoted $s\restrictedto n$, as the partial state (for \textbf{I} or \textbf{II} resp.) following $s$ up to turn $n$ of $s$.
    \end{enumerate}
\end{definition}

\begin{definition}
\label{def:state.last}
    If $s = (A_0,a_1,B_0,\dots,A_n)$ is a state of rank $n$ ending with a play by \textbf{I}, then $\last(s) := A_n$. If $s = (A_0,a_1,B_0,\dots,A_n,a_{n+1},B_n)$ is a state of rank $n$ ending with a play by \textbf{II}, then $\last(s) := B_n$. We also define $\last(\emptyset) := A$. 
\end{definition}

\subsection{Basic Properties}
Let $(\cal{R},\leq,r)$ be a \textbf{wA2}-space. We let $\bcal{KR}$ denote the set of all Kastanas Ramsey subsets of $\cal{R}$, and let $\bar{\bcal{KR}}$ be the set of all subsets of $\cal{R}$ whose complement is Kastanas Ramsey. In this section, we study various set-theoretic properties of $\bcal{KR}$. We state some positive results that are analogous to those of strategically Ramsey sets presented in \cite{R08}. The proofs are inspired by those shown in the same article.

\begin{lemma}
\label{lem:kr.general.dichotomy}
    Let $(\cal{R},\leq,r)$ be a \textbf{wA2}-space. Let $\X_n \subseteq \cal{R}$ for each $n$, and let $\X := \bigcap_{n<\omega} \X_n$. For any $A \in \cal{R}$ and $a \in \cal{AR}\restrictedto A$, there exists some $B \in [a,A]$ such that one of the following must hold:
    \begin{enumerate}
        \item \textbf{I} has a strategy in $K[a,B]$ to reach $\X$.

        \item \textbf{II} has a strategy $\tau$ in $K[a,B]$ such that the following holds: For any total state $s$ following $\tau$, there exists some $n$ such that if $a(s\restrictedto n) = a_n$ and $\last(s\restrictedto n) = B_n$, then \textbf{I} has no strategy in $K[a_n,B_n]$ to reach $\X_n$.
    \end{enumerate}
\end{lemma}

See also Lemma 4, \cite{R08}.

\begin{proof}
    Let $\{\X_n\}_{n<\omega}$ be a countable family of subsets of $\cal{R}$, and let $\X := \bigcap_{n<\omega} \X_n$. Fix any $A \in \cal{R}$ and $a \in \cal{AR}\restrictedto A$, and assume that (2) fails for all $B \in [a,A]$. In particular, applying $n = 0$ to the negation of (2), \textbf{I} has a strategy in $K[a,B]$ to reach $\X_0$ for all $B \in [a,A]$.

    If $B \in [a,A]$ and $b \in \cal{AR}\restrictedto[a,B]$, say that ``(2) holds for $(b,B)$'' if \textbf{II} has a strategy $\tau$ in $K[b,B]$ such that, for any total state $s$ following $\tau$, there exists some $n \geq \lh(b) - \lh(a)$ such that if $\last(s\restrictedto(n - \lh(b) + \lh(a))) = B_n$ and $a(s\restrictedto(n - \lh(b) + \lh(a))) = a_n$, then \textbf{I} has no strategy in $K[a_n,B_n]$ to reach $\X_n$. Note that this would also imply that \textbf{I} has no strategy in $K[a_n,C]$ to reach $\X_n$ for all $C \in [a_n,B_n]$. 

    \begin{claim}
        For all $B \in [a,A]$ and $b \in \cal{AR}\restrictedto[a,B]$, (2) holds for $(b,B)$ iff for all $A' \in [b,B]$, there exists some $b' \in r_{\lh(b)+1}[b,A']$ and $B' \in [b',A']$ such that (2) holds for $(b',B')$.
    \end{claim}

    \begin{midproof}
        \underline{$\implies$:} Let $\tau$ be the strategy in $K[b,B]$ witnessing that (2) holds for $(b,B)$. Given any $A' \in [b,B]$, consider the play where \textbf{I} begins with $A$, and player \textbf{II} responds with $b' \in r_{\lh(b)+1}[b,B]$ and $B' \in [b',A']$ according to $\tau$. The restriction of the strategy $\tau$ to $K[b',B']$ is a strategy witnessing that (2) holds for $(b',B')$.

        \underline{$\impliedby$:} Suppose that for all $A' \in [b,B]$, there exists some $b' \in r_{\lh(b)+1}[b',A']$, $B' \in [b',A']$ and strategy $\tau_{A'}$ witnessing that (2) holds for $(b',B')$. Define the strategy $\tau$ in $K[b,B]$ such that if \textbf{I} begins with $A'$, then \textbf{II} responds with $b'$ and $B'$, then continue according to $\tau_{A'}$. This gives a strategy witnessing that (2) holds for $(b,B)$.
    \end{midproof}
    
    Let $\O_\emptyset$ denote the set of all $a' \in r_{\lh(a)+1}[a,A]$ such that for all $B \in [a',A]$, (2) does not hold for $(a',B)$. Since we assumed that (2) fails for all $B \in [a,A]$, by the previous claim there exists some $A_\emptyset \in [a,A]$ such that $r_{\lh(a)+1}[a,A'] \subseteq \O_\emptyset$. As stated in the first paragraph, \textbf{I} has a strategy $\sigma_\emptyset$ in $K[a,A_\emptyset]$ to reach $\X_0$. To finish the proof, we shall construct a strategy $\sigma$ for \textbf{I} in $K[a,A_\emptyset]$ to reach $\X$.
    
    For the rest of this proof, all states are assumed to be for \textbf{II}. Let $\sigma(\emptyset) := \sigma_\emptyset(\emptyset)$. Now suppose for each state $s$ of $K[a,A_\emptyset]$ following $\sigma$ of rank $n$, we have the following defined:
    \begin{enumerate}
        \item If $n > 0$ and $s' = s\restrictedto(n - 1)$, then $A_s \in [a(s),A_{s'}]$.

        \item $\sigma_s$ is a strategy for \textbf{I} in $K[a(s),A_s]$ to reach $\X_n$.

        \item If $i \leq n$, then $t_s^i$ is a state of $K[a(s\restrictedto i),A_{s\restrictedto i}]$ following $\sigma_{s\restrictedto i}$ of rank $n - i$, and $a(t_s^i) = a(s\restrictedto i)$.

        \item $\sigma_{s\restrictedto(i+1)}(t_s^{i+1}) \in [a_s,\sigma_{s\restrictedto i}(t_s^i)]$, and $\sigma(s) \in [a(s),\sigma_s(t_s^n)]$.

        \item For all $b \in r_{\lh(a(s))+1}[a(s),A_s]$ and $B \in [b,A_s]$, (2) does not hold in $(b,B)$.
    \end{enumerate}
    Now let $s$ be a state of $K[a,A_\emptyset]$ following $\sigma$ of rank $n + 1$. We may write $s = (s\restrictedto n)^\frown(a(s),B_s)$. Since $B_s \in [a(s),\sigma(s\restrictedto n)] \subseteq [a(s),\sigma_\emptyset(t_{s\restrictedto n}^0)]$, we may define $t_s^0 := {t_{s\restrictedto n}^0}^\frown(a(s),B_s)$, which is a legal state in $K[a,A_\emptyset]$ following $\sigma_\emptyset$. We have $\sigma_\emptyset(t_s^0) \in [a(s),B_s] \subseteq [a(s),\sigma_{s\restrictedto 1}(t_{s\restrictedto n}^1)]$, so we may define $t_s^1 := {t_{s\restrictedto n}^1}^\frown(a(s),\sigma_\emptyset(t_s^0))$. This again, gives us a legal state in $K[a(s\restrictedto 1),A_{s\restrictedto 1}]$ following $\sigma_{s\restrictedto 1}$. We may repeat this process to give us states $t_s^i$ for $i \leq n$. We let $A_s' := \sigma_{s\restrictedto n}(t_s^n) \in [a(s\restrictedto n),A_{s\restrictedto n}]$. 

    Let $\O_s$ be the set of all $a' \in r_{\lh(a(s))+1}[a(s),A_s']$ such that there exists some $B \in [a(s),A_s']$ in which (2) holds for $(a',B)$. By (5) of the induction hypothesis and the claim, we may obtain $A_s \in [a(s),A_s']$ such that $r_{\lh(a(s))+1}[a(s),A_s] \subseteq \O_s^c$, and \textbf{I} has a winning strategy $\sigma_s$ in $K[a(s),A_s]$ to reach $\X_{n+1}$. Define $\sigma(s) := \sigma_s(\emptyset)$. This is indeed a legal move, as:
    \begin{align*}
        \sigma(s) \in [a(s),A_s] \subseteq [a(s),A_s'] \subseteq [a(s),A_{s\restrictedto n}] \subseteq [a(s),B_s].
    \end{align*}
    This completes the induction. We see that $\sigma$ is indeed a strategy for \textbf{I} in $K[a,A_\emptyset]$ to reach $\X$ - if $s$ is a total state of $K[a,A_\emptyset]$ following $\sigma$, then $A(s) = A(t_s^n) \in \X_n$ for all $n$, so $A(s) \in \bigcap_{n<\omega} \X_n = \X$. This completes the proof.
\end{proof}

\begin{proposition}
\label{prop:kr.closed.under.countable.unions}
    For any \textbf{wA2}-space $(\cal{R},\leq,r)$, $\bcal{KR}$ is closed under countable unions.
\end{proposition}

See also Theorem 9, \cite{R08}.

\begin{proof}
    Let $\{\X_n\}_{n<\omega}$ be a countable family of subsets of $\cal{R}$, and let $\X := \bigcup_{n<\omega} \X_n$. Fix any $A \in \cal{R}$ and $a \in \cal{AR}\restrictedto A$. If there exists some $B \in [a,A]$ such that \textbf{II} has a strategy in $K[a,B]$ to reach $\X$, then we're done, so assume otherwise. Consider applying Lemma \ref{lem:kr.general.dichotomy} to $\X^c = \bigcap_{n<\omega} \X_n^c$. We claim that (2) fails for all $(a,B)$ where $B \in [a,A]$, so by the same lemma, \textbf{I} has a strategy in $K[a,A]$ to reach $\X^c$. Indeed, otherwise let $\tau$ be a strategy in $K[a,B]$ witnessing that (2) holds for $(a,B)$. Player \textbf{II} shall follow $\tau$ until they reach some turn $n$, ending with \textbf{II} playing $(a_n,B_n')$, such that \textbf{I} has no strategy in $K[a_n,B_n']$ to reach $\X_n^c$. Since $\X_n$ is Kastanas Ramsey, \textbf{II} may instead play $(a_n,B_n)$ in the last turn, where $B_n \in [a_n,B_n']$, such that \textbf{II} has a strategy in $K[a_n,B_n]$ to reach $\X_n$. Afterwards, \textbf{II} follows this strategy to reach $\X_n$. Since $\X_n \subseteq \X$, this constitutes a strategy for \textbf{II} in $K[a,B]$ to reach $\X$, contradicting our assumption. 
\end{proof}

We now turn our attention to some negative results. 

\begin{definition}
    Let $(\cal{R},\leq,r)$ be a \textbf{wA2}-space, and let $\O \subseteq \cal{AR}$.
    \begin{enumerate}
        \item $\O$ is \emph{$(a,A)$-biasymptotic}, where $A \in \cal{R}$ and $a \in \cal{AR}\restrictedto A$, if for all $B \in [\depth_A(a),A]$, $r_{\lh(a)+1}[a,B] \cap \O \neq \emptyset$ and $r_{\lh(a)+1}[a,B] \cap \O^c \neq \emptyset$.

        \item $\O$ is \emph{biasymptotic} if $\O$ is $(a,A)$-biasymptotic for all $A \in \cal{R}$ and $a \in \cal{AR}\restrictedto A$.
    \end{enumerate}
\end{definition}

Thus, the assertion that $\cal{R}$ does not satisfy \textbf{A4} is equivalent to the assertion that there exists an $(a,A)$-biasymptotic set for some $A \in \cal{R}$ and $a \in \cal{AR}$. We illustrate some examples here. 

\begin{example}[Infinite block sequences $\FIN_{\pm k}^{[\infty]}$]
    Recall that for each $x \in \FIN_{\pm k}$, we defined:
    \begin{align*}
        n_x := \min\{n < \omega : x(n) = \pm k\},
    \end{align*}
    and:
    \begin{align*}
        Y := \{x \in \FIN_{\pm k} : x(n_x) = k\}.
    \end{align*}
    Then, for all $A \in \FIN_{\pm k}^{[\infty]}$, $\c{A} \cap Y \neq \emptyset$ and $\c{A} \cap Y^c \neq \emptyset$. Thus, the set:
    \begin{align*}
        \O := \{a \in \FIN_k^{[<\infty]} : a = (x_0,\dots,x_n) \wedge x_n \in Y\}.
    \end{align*}
    is a biasymptotic set.
\end{example}

\begin{example}[Countable vector space $E^{[\infty]}$]
\label{ex:biasymptotic.set.vector.spaces}
    Let $\mathbb{F}$ be a field such that $|F| > 2$, and let $E$ be an $\mathbb{F}$-vector space of dimension $\aleph_0$. We defined the set:
    \begin{align*}
        Y := \{x \in E : \text{$x = e_n + y$ for some $n$ and $e_n < y$}\}.
    \end{align*}
    We have that $\c{A} \cap Y \neq \emptyset$ and $\c{A} \cap Y^c \neq \emptyset$ for all $A \in E^{[<\infty]}$. Thus, the set:
    \begin{align*}
        \{a \in E^{[<\infty]} : a = (x_0,\dots,x_n) \wedge x_n \in Y\}
    \end{align*}
    is biasymptotic.
\end{example}

\begin{proposition}
    Let $(\cal{R},\leq,r)$ be a \textbf{wA2}-space. If \textbf{A4} fails, then there exists some $\X \in \bcal{KR} \cap \bar{\bcal{KR}}$ which is not Ramsey.
\end{proposition}

\begin{proof}
    Let $\O$ be an $(a,A)$-biasymptotic set for some $A \in \cal{R}$ and $a \in \cal{AR}\restrictedto A$. Define:
    \begin{align*}
        \X := \{C \in \cal{R} : a \sqsubseteq C \wedge r_{\lh(a)+1}(C) \in \O\}.
    \end{align*}
    Since $\O$ is $(a,A)$-biasymptotic, for any $B \in [a,A]$, there exists some $C \in [a,B]$ such that $r_{\lh(a)+1}(C) \in \O$, and some $C' \in [a,B]$ such that $r_{\lh(a)+1}(C) \notin \O$. Consequently, $[a,B] \cap \X \neq \emptyset$ and $[a,B] \cap \X^c \neq \emptyset$, so $\X$ is not Ramsey as $B$ is arbitrary. However, $\X$ is a countable union of clopen sets, so it is Borel (under the metrisable topology). By the Borel determinacy for $\R^\omega$, the game $K[a,A]$ to reach $\X$ or $\X^c$ is always determined, so $\X \in \bcal{KR} \cap \bar{\bcal{KR}}$.
\end{proof}

\begin{proposition}
\label{prop:kr.not.symmetric}
    Let $(\cal{R},\leq,r)$ be a \textbf{wA2}-space, and assume that it has a biasymptotic set. If $\bcal{KR} \neq \Po(\cal{R})$, then $\bcal{KR}$ is not closed under complements.
\end{proposition}

\begin{proof}
    Fix a biasymptotic set $\O$, and let $\X \subseteq \cal{R}$ be not Kastanas Ramsey. Define two sets as follows:
    \begin{align*}
        \X_0 &:= \{C \in \X : \forall m \, \exists n \geq m[r_n(C) \in \O]\}, \\
        \X_1 &:= \{C \in \X : \exists m \, \forall n \geq m[r_n(C) \notin \O]\}.
    \end{align*}
    Observe that both $\X_0^c$ and $\X_1^c$ are Kastanas Ramsey: Indeed, for any $A \in \cal{R}$ and $a \in \cal{AR}\restrictedto A$, \textbf{II} has a winning strategy in $K[a,A]$ to reach $\X_0^c$ by playing $a_n \notin \O$ for all $n$, and \textbf{II} also has a winning strategy in $K[a,A]$ to reach $\X_1^c$ by playing $a_n \in \O$ for all $n$. On the other hand, we have that $\X_0 \cup \X_1 = \X$, so if both $\X_0$ and $\X_1$ are Kastanas Ramsey, then so is $\X$ by Proposition \ref{prop:kr.closed.under.countable.unions}, a contradiction. Thus, at least one of $\X_0$ or $\X_1$ witnesses that $\bcal{KR}$ is not closed under complements.
\end{proof}

\begin{proposition}
\label{prop:kr.not.closed.under.intersections}
    Let $(\cal{R},\leq,r)$ be a \textbf{wA2}-space, and assume that it has a biasymptotic set. If $\bcal{KR} \neq \Po(\cal{R})$, then $\bcal{KR}$ is not closed under finite intersections.
\end{proposition}

\begin{proof}
    Fix a biasymptotic set $\O$, and let $\X \subseteq \cal{R}$ be a Kastanas Ramsey set in which $\X^c$ is not Kastanas Ramsey. Define two sets as follows:
    \begin{align*}
        \X_0 &:= \{C \in \X : \forall m \, \exists n \geq m[r_n(C) \in \O]\}, \\
        \X_1 &:= \{C \in \X : \exists m \, \forall n \geq m[r_n(C) \notin \O]\}.
    \end{align*}
    By the same argument as in Proposition \ref{prop:kr.not.symmetric}, $\X_0^c$ and $\X_1^c$ are Kastanas Ramsey. However, $\X^c = \X_0^c \cap \X_1^c$ is not.
\end{proof}

\subsection{Kastanas Ramsey sets in topological Ramsey spaces} We shall now give a proof of Theorem \ref{thm:kr.iff.r}, which is split into a proof of two different propositions. The first proposition is the following:

\begin{proposition}[Proposition 4.2, \cite{CD23}]
\label{prop:kr.equals.r.for.I}
    Let $(\cal{R},\leq,r)$ be an \textbf{A2}-space. For every $\X \subseteq \cal{R}$, $A \in \cal{R}$ and $a \in \cal{AR}\restrictedto A$, \textbf{I} has a strategy in $K[a,A]$ to reach $\X$ iff $[a,B] \subseteq \X$ for some $B \in [a,A]$.
\end{proposition}

We remark that the proof in \cite{CD23} assumes that $(\cal{R},\leq,r)$ satisfies the following property: If $a \in \cal{AR}\restrictedto A$, and $b \sqsubseteq a$ but $b \neq a$, then $\depth_A(b) < \depth_A(a)$. While it is not true that all spaces satisfying \textbf{A1}-\textbf{A4} would also satisfy such a property, the gap may be fixed with a careful enumeration of elements of $\cal{AR}$. We omit the details.

The second proposition is the following:

\begin{proposition}
\label{prop:kr.II.strategy.gives.I}
    Suppose that $(\cal{R},\leq,r)$ satisfies \textbf{A1}-\textbf{A4}. For every $\X \subseteq \cal{R}$, $A \in \cal{R}$ and $a \in \cal{AR}\restrictedto A$, if \textbf{II} has a strategy in $K[a,A]$ to reach $\X$, then \textbf{I} has a strategy in $K[a,A]$ to reach $\X$.
\end{proposition}

A proof of Proposition \ref{prop:kr.II.strategy.gives.I} for selective topological Ramsey spaces (Definition 5.4, \cite{CD23}) was provided in \cite{CD23}. We shall use the idea presented in \cite{DMU12} to instead prove Lemma 5.5 of \cite{CD23} using a semiselectivity argument. 

\begin{definition}
\label{def:semiselective}
    Let $A \in \cal{R}$ and $a \in \cal{AR}\restrictedto A$.
    \begin{enumerate}
        \item A family of subsets $\vec{\D} = \{\D_b\}_{b \in \cal{AR}\restrictedto[a,A]}$ is \emph{dense open below $[a,A]$} if for all $b \in \cal{AR}\restrictedto[a,A]$, $\D_b$ is a $\leq$-downward closed subset of $[b,A]$, and for all $B \in [b,A]$, there exists some $C \in [b,B]$ such that $C \in \D_b$.
        
        \item Let $\vec{\D} = \{\D_b\}_{b \in \cal{AR}\restrictedto[a,A]}$ be dense open below $[a,A]$. We say that $B \in [a,A]$ \emph{diagonalises} $\vec{\D}$ if for all $b \in \cal{AR}\restrictedto[a,B]$, there exists some $A_b \in \D_b$ such that $[b,B] \subseteq [b,A_b]$.
    \end{enumerate}
\end{definition}

\begin{lemma}
\label{lem:R.is.semiselective}
    If $(\cal{R},\leq,r)$ is an \textbf{A2}-space, then every family of subsets $\vec{\D} = \{\D_b\}_{b \in \cal{AR}\restrictedto[a,A]}$ which is dense open below $[a,A]$ has a diagonalisation.
\end{lemma}

In other words, Lemma \ref{lem:R.is.semiselective} asserts that $\cal{R}$ is a ``semiselective coideal''.

\begin{proof}
    Fix some $A \in \cal{R}$ and $a \in \cal{AR}\restrictedto A$. Suppose that $\vec{\D} = \{\D_b\}_{b \in \cal{AR}\restrictedto[a,A]}$ is dense open below $[a,A]$. We shall define a fusion sequence $(A_n)_{n<\omega}$ in $[a,A]$, with $a_{n+1} := r_{\lh(a)+n+1}(A_n)$, such that $A_{n+1} \in [a_{n+1},A_n]$: Let $A_0 := A$, and suppose that $A_n$ has been defined. Let $\{b_i : i < N\}$ enumerate the set of all $b \in \cal{AR}\restrictedto A_n$ such that $a \sqsubseteq b$ and $b \leq_\fin a_{n+1}$. Let $A_{n+1}^0 := A_n$. If $A_{n+1}^i \in [a_{n+1},A_n]$ has been defined, let $B_{n+1} \in \D_{b_i}$ be such that $B_{n+1} \in [b_i,A_{n+1}^i]$, which exists as $\D_{b_i}$ is dense open in $[b_i,A]$. By \textbf{A3}, we then let $A_{n+1}^{i+1} \in [a_{n+1},A_n^i]$ such that $[b_i,A_{n+1}^{i+1}] \subseteq [b_i,B_{n+1}]$. We complete the induction by letting $A_{n+1} := A_{n+1}^N$. Let $B$ be the limit of the fusion sequence $(A_n)_{n<\omega}$, and we have that $B$ diagonalises $\vec{\D}$. 
\end{proof}

We are now ready to prove Proposition \ref{prop:kr.II.strategy.gives.I}.

\begin{lemma}
\label{lem:diagonalise.suitable.functions}
    Let $(\cal{R},\leq,r)$ be a closed triple satisfying \textbf{A1}-\textbf{A4}. Suppose that $f : [a,A] \to r_{\lh(a)+1}[a,A]$ and $g : [a,A] \to [a,A]$ are two functions such that for all $B \leq A$:
    \begin{enumerate}
        \item $f(B) \in r_{\lh(a)+1}[a,B]$.

        \item $g(B) \in [f(B),B]$.
    \end{enumerate}
    We also say that these two functions $f,g$ are \emph{suitable} in $[a,A]$. Then there exists some $E_{f,g} \in [a,A]$ such that for all $b \in r_{\lh(a)+1}[a,E_{f,g}]$, there exists some $B \in [a,A]$ such that $f(B) = b$ and $[b,E_{f,g}] \subseteq [b,g(B)]$. 
\end{lemma}

\begin{proof}
    For each $b \in r_{\lh(a)+1}[a,A]$, we define:
    \begin{align*}
        \D_{b,0} := \{D \in [b,A] : \exists B \in [a,A] \text{ s.t. } f(B) = b \wedge D \in [b,g(B)]\}, \\
        \D_{b,1} := \{D \in [b,A] : \forall C \in [a,A], \, g(C) \in [b,D] \to g(C) \notin \D_{b,0}\}.
    \end{align*}
    Let $\D_b := \D_{b,0} \cup \D_{b,1}$. Observe that $\D_b$ is dense open in $[b,A]$: Clearly both $\D_{b,0}$ and $\D_{b,1}$ are open. If $D \in [b,A]$ and $D \notin \D_{b,1}$, then there exists some $C \in [a,A]$ such that $g(C) \in [b,D]$. Then $g(C) \leq D$ and $g(C) \in \D_{b,0}$, so $\D_b$ is dense.

    By Lemma \ref{lem:R.is.semiselective}, there exists some $D \in [a,A]$ diagonalising $(\D_b)_{b \in \cal{AR}\restrictedto[a,A]}$. Now let:
    \begin{align*}
        \O_0 := \{b \in r_{\lh(a)+1}[a,D] : \exists B \in \D_{b,0} \, [b,D] \subseteq [b,B]\}, \\
        \O_1 := \{b \in r_{\lh(a)+1}[a,D] : \exists B \in \D_{b,1} \, [b,D] \subseteq [b,B]\}.
    \end{align*}
    By \textbf{A4}, there exists some $E_{f,g} \in [a,D]$ such that $r_{\lh(a)+1}[a,E_{f,g}] \subseteq \O_0$ or $r_{\lh(a)+1}[a,E_{f,g}] \subseteq \O_1$. However, we see that the latter case is not possible: In this case, we let $b := f(E_{f,g})$. Then $b \in r_{\lh(a)+1}[a,E_{f,g}] \subseteq \O_1$, so let $B \in \D_{b,1}$ such that $[b,D] \subseteq [b,B]$. Then $g(E_{f,g}) \in [b,E_{f,g}] \subseteq [b,D] \subseteq [b,B]$, so $g(E_{f,g}) \notin \D_{b,0}$. But $g(E_{f,g}) \in [b,B]$, so this implies that $f(E_{f,g}) \neq b$, a contradiction.

    We shall show that $E_{f,g}$ works. Let $b \in r_{\lh(a)+1}[a,E_{f,g}]$. Then $b \in \O_0$, so there exists some $B \in \D_{b,0}$ such that $f(B) = b$ and $[b,E_{f,g}] \subseteq [b,D] \subseteq [b,g(B)]$, as desired.
\end{proof}

\begin{proof}[Proof of Proposition \ref{prop:kr.II.strategy.gives.I}]
    Let $\sigma$ be a strategy for \textbf{II} in $K[a,A]$ to reach $\X$. We shall construct a strategy $\tau$ for \textbf{I} in $K[a,A]$ to reach $\X$ as follows: We shall assign a state $s$ (for \textbf{II}) in $K[a,A]$ following $\tau$, to a state $t_s$ (for \textbf{II) in $K[a,A]$}, following $\sigma$, such that:
    \begin{enumerate}
        \item $a(s) = a(t_s)$.

        \item $\last(s) \leq \last(t_s)$. 

        \item If $s' \sqsubseteq s$, then $t_{s'} \sqsubseteq t_s$.
    \end{enumerate}
    We begin by defining $t_\emptyset := \emptyset$. Now suppose that $s$ is a state (for \textbf{II}) in $K[a,A]$ following $\tau$ so far. We define the functions $f_s,g_s$ by stipulating that for all $B \in [a(t_s),\last(t_s)]$, $(f_s(B),g_s(B)) := \sigma({t_s}^\frown B)$. Observe that $f_s,g_s$ are suitable in $[a(s),\last(s)]$ (when restricted to $[a(s),\last(s)]$), so by Lemma \ref{lem:diagonalise.suitable.functions} there exists some $E_s \in [a(s),\last(s)]$ such that for all $b \in r_{\lh(a(s))+1}[a(s),E_s]$, there exists some $B_{s,b} \in [a(s),\last(s)]$ such that $f_s(B_{s,b}) = b$ and $[b,E_s] \subseteq [b,g(B_{s,b})]$. Thus, we define $\tau(s) := E_s \in [a(s),\last(s)]$, and for all $b \in r_{\lh(a(s))+1}[a(s),E_s]$ and $C \in [b,E_s]$, define:
    \begin{align*}
        t_{s^\frown(b,C)} := {t_s}^\frown(b,g(B_{s,b})).
    \end{align*}
    Clearly, (1) and (3) of the induction hypothesis are satisfied. (2) is also satisfied, as:
    \begin{align*}
        \last(s) = C \in [b,E_s] \subseteq [b,g(B_{s,b})].
    \end{align*}
    This completes the induction. Since every total state following $\tau$ corresponds to a total state following $\sigma$ with the same outcome, $\tau$ is a strategy for \textbf{I} to reach $\X$.
\end{proof}

Combined with Proposition \ref{prop:kr.closed.under.countable.unions}, we get:

\begin{corollary}
\label{cor:measurable.is.ramsey}
    Let $(\cal{R},\leq,r)$ be a closed triple satisfying \textbf{A1}-\textbf{A4}. Then the set of (Kastanas) Ramsey subsets of $\cal{R}$ forms a $\sigma$-algebra.
\end{corollary}

\section{Kastanas Ramsey Sets and the Projective Hierarchy}
Given a \textbf{wA2}-space $(\cal{R},\leq,r)$, if $\cal{AR}$ is countable then the metrisable topology is Polish, allowing us to discuss the projective hierarchy on $\cal{R}$. We discuss some relationships between Kastanas Ramsey sets and sets in the projective hierarchy.

\subsection{Projective Hierarchy} In this section, we introduce prove a general relationship between Kastanas Ramsey sets and sets in the projective hierarchy, which, in particular, gives a proof of Theorem \ref{thm:analytic.is.kr}.

Given a \textbf{wA2}-space $(\cal{R},\leq,r)$, we shall construct another \textbf{wA2}-space $(\cal{R} \times 2^\omega,\preceq,r)$ as follows:
\begin{enumerate}
    \item $\cal{A}(\cal{R} \times 2^\omega) := \bigcup_{n<\omega} \cal{AR}_n \times 2^n$.

    \item Given $(A,u) \in \cal{R} \times 2^\omega$, let $r_n(A,x) := (r_n(A),u\restrictedto n)$. Thus, if $(a,p) \in \cal{A}(\cal{R} \times 2^\omega)$, then $\lh(a,p) = \lh(a) = |p|$.

    \item We define a $\preceq_\fin$ on $\cal{A}(\cal{R} \times 2^\omega)$ by stipulating that $(a,p) \preceq_\fin (b,q)$ iff $a \leq_\fin b$. 

    \item Given $(A,u),(B,v) \in \cal{R} \times 2^\omega$, we write:
    \begin{align*}
        (A,u) \preceq (B,v) \iff \forall n \, \exists m[r_n(A,u) \preceq_\fin r_m(B,v)].
    \end{align*}
\end{enumerate}

We remark that $\preceq$ is \emph{never} a partial order. For instance, if $a \in \cal{AR}_2$, then $(a,(0,1)) \preceq_\fin (a,(1,0))$ and $(a,(1,0)) \preceq_\fin (a,(0,1))$, but $(a,(0,1)) \neq (a,(1,0))$.

\begin{lemma}
    Let $(\cal{R},\leq,r)$ be a \textbf{wA2}-space (resp. \textbf{A2}-space). Then the closed triple $(\cal{R} \times 2^\omega,\preceq,r)$ defined above is a \textbf{wA2}-space (resp. \textbf{A2}-space) which has a biasymptotic set.
\end{lemma}

\begin{proof}
    It is easy to verify that $(\cal{R} \times 2^\omega,\preceq,r)$ satisfies \textbf{A1}, \textbf{wA2} (resp. \textbf{A2}) and \textbf{A3}. A biasymptotic set would be:
    \begin{align*}
        \O = \{(a,p) \in \cal{A}(\cal{R} \times 2^\omega) : |p| > 0 \wedge p(|p|-1) = 1\}.
    \end{align*}
\end{proof}

Let $\pi_0 : \cal{R} \times 2^\omega \to \cal{R}$ denote the projection to the first coordinate, which is a surjective map which respects $\leq$ (i.e. if $(A,p) \preceq (B,q)$ then $\pi_0(A,p) \leq \pi_0(B,q)$). We also use $\vec{0}$ to denote the infinite tuple of zeroes $(0,0,0,\dots) \in 2^\omega$.

\begin{lemma}
    Let $(\cal{R},\leq,r)$ be a \textbf{wA2}-space. Let $\cal{C} \subseteq \cal{R} \times 2^\omega$ be a subset. Let $A \in \cal{R}$ and $a \in \cal{AR}\restrictedto A$. If \textbf{II} has a strategy in $K[(a,p),(A,\vec{0})]$ to reach $\cal{C}$ for some $p \in 2^{\lh(a)}$, then \textbf{II} has a strategy in $K[a,A]$ to reach $\pi_0[\cal{C}]$.
\end{lemma}

\begin{proof}
    If $\sigma$ is a strategy for \textbf{II} in $K[(a,p),(A,\vec{0})]$ to reach $\cal{C}$, then the strategy $\tau$ for \textbf{II} in $K[a,A]$ defined by $\tau(A_0,\dots,A_{n-1}) := (b,B)$, where $\sigma((A_0,\vec{0}),\dots,(A_{n-1},\vec{0})) = ((b,p),(B,u))$ for some $p,u$, is a strategy to reach $\pi_0[\cal{C}]$.
\end{proof}

\begin{lemma}
    Let $(\cal{R},\leq,r)$ be a \textbf{wA2}-space. Let $\cal{C} \subseteq \cal{R} \times 2^\omega$ be a subset. Let $A \in \cal{R}$ and $a \in \cal{AR}\restrictedto A$. If for all $p \in 2^{\lh(a)}$ and $B \in [a,A]$, there exists some $C \in [a,B]$ such that \textbf{I} has a strategy in $K[(a,p),(C,\vec{0})]$ to reach $\cal{C}^c$, then \textbf{I} has a strategy in $K[a,A]$ to reach $\pi_0[\cal{C}]^c$.
\end{lemma}

\begin{proof}
    We shall construct a strategy $\tau$ for \textbf{I} in $K[a,A]$ to reach $\pi_0[\cal{C}]^c$. Let $\{p_k : k < 2^{\lh(a)}\}$ enumerate the set $2^{\lh(a)}$. We define a decreasing sequence $C_0 \geq C_1 \geq \cdots C_{2^{\lh(a)}-1}^0$ as follows: Let $\sigma_{p_0}$ be the strategy for \textbf{I} in $K[(a,p_0),(A,\vec{0})]$ to reach $\cal{C}^c$, and let $C_0 := \sigma_{p_0}(\emptyset)$. We also define $B_{p_0} := A$. Suppose that $A_k^0$ has been defined. By the hypothesis, there exists some $B_{p_{k+1}} \in [a,C_k]$ such that \textbf{I} has a strategy in $K[(a,p_{k+1}),(B_{p_{k+1}},\vec{0})]$ to reach $\cal{C}^c$. Now let $C_{k+1} := \sigma_{p_{k+1}}(\emptyset)$. Having constructed the above sequence, we now define $\tau(\emptyset) := C_{2^{\lh(a)}-1}$. 

    Note that we may assume that for all partial states $t$ of $K[(a,p),(B_p,\vec{0})]$ for \textbf{II}, $\sigma_p(t) = (B,\vec{0})$ for some $B \in \cal{R}$. Now let $s$ be a partial state of $K[a,A]$ for \textbf{II} following $\tau$ so far, and $\rank(s) = n$, and assume that $\tau(s)$ has been defined. Suppose for the induction hypothesis that we have a set $\{t_{s,q} : q \in 2^{\lh(a)+n}\}$ such that:
    \begin{enumerate}
        \item $t_{s,q}$ is a partial state of $K[(a,q\restrictedto\lh(a)),(B_{q\restrictedto\lh(a)},\vec{0})]$ following $\sigma_{q\restrictedto\lh(a)}$ with $a(t_{s,q}) = (a(s),q)$.

        \item If $\sigma_{q\restrictedto\lh(a)}(t_{s,q}) = (A_{s,q},\vec{0})$, then $\tau(s) \in [a(s),A_{s,q}]$. 
    \end{enumerate}
    Note that for the base case, for each $p \in 2^{\lh(k)}$ we let $t_{\emptyset,a}$ be the empty state of the game $K[(a,p),(B_p,\vec{0})]$. For each $a_{n+1} \in r_{\lh(a(s))+1}[a(s),\tau(s)]$ and $B_n \in [a_{n+1},\tau(s)]$, Let $\{q_k : k < 2^{\lh(a)+n+1}\}$ enumerate the set $2^{\lh(a)+n+1}$, and for each $k$ we let $q_k' := q_k\restrictedto(\lh(a)+n)$ and $p_k := q_k\restrictedto\lh(a)$ (which differs from the enumeration in the first paragraph, but it doesn't matter). We define a decreasing sequence $D_0 \geq D_1 \geq \cdots \geq D_{2^{\lh(a)+n+1}-1}$ as follows: Let $D_0 := \sigma_{p_0}({t_{s,q_0'}}^\frown((a_{n+1},q_0),(B_n,\vec{0})))$. Assuming that $D_k$ has been defined, we let $D_{k+1} := \sigma_{p_k}({t_{s,q_k'}}^\frown((a_{n+1},q_k),(D_k,\vec{0})))$. Note that by the induction hypothesis, all the partial states listed above are legal. We conclude the construction of $\tau$ by asserting that $\tau(s^\frown(a_{n+1},B_n)) := D_{2^{\lh(a)+n+1}-1}$.

    We shall now show that $\tau$ is a strategy for \textbf{I} in $K[a,A]$ to reach $\pi_0[\cal{C}]^c$. If not, then there exists some complete play $K[a,A]$ following $\tau$ such that \textbf{II} plays $(a_1,B_0,a_2,B_1,\dots)$, and $C := \lim_{n\to\infty} a_n \in \pi_0[\cal{C}]$. In particular, we have that $(C,x) \in \cal{C}$ for some $x \in 2^\omega$. For each $n$, let $q_n := x\restrictedto(\lh(a)+n)$ and let $p := q_0$. By our construction of $\tau$, there exists a complete play of the game $K[(a,p),(B_p,\vec{0})]$ following $\sigma_p$ such that \textbf{II} plays $((a_1,q_1),(B_0,\vec{0}),(a_2,q_2),(B_1,\vec{0}),\dots)$. But since $\sigma_p$ is a strategy for \textbf{I} in $K[(a,p),(A,\vec{0})]$ to reach $\cal{C}^c$, we have that $(C,x) \in \cal{C}^c$, a contradiction.
\end{proof}

These two lemmas lead us to the following results.

\begin{theorem}
    Let $(\cal{R},\leq,r)$ be a \textbf{wA2}-space. If $\cal{C} \subseteq \cal{R} \times 2^\omega$ is Kastanas Ramsey, then $\pi_0[\cal{C}] \subseteq \cal{R}$ is Kastanas Ramsey.
\end{theorem}

\begin{theorem}
\label{thm:kr.and.projection}
    Let $(\cal{R},\leq,r)$ be a \textbf{wA2}-space, and assume that $\cal{AR}$ is countable. 
    \begin{enumerate}
        \item Every analytic subset of $\cal{R}$ is Kastanas Ramsey.

        \item If every coanalytic subset of $\cal{R} \times 2^\omega$ is Kastanas Ramsey, then every $\mathbf{\Sigma}_2^1$ subset of $\cal{R}$ is Kastanas Ramsey. More generally, for every $n \geq 1$, if every $\mathbf{\Pi}_n^1$ subset of $\cal{R} \times 2^\omega$ is Kastanas Ramsey, then every $\mathbf{\Sigma}_{n+1}^1$ subset of $\cal{R}$ is Kastanas Ramsey.
    \end{enumerate}
\end{theorem}

See also Theorem IV.4.14, \cite{AT05}. We remark that one may alternatively prove that every analytic subset of $\cal{R}$ is Kastanas Ramsey using Lemma \ref{lem:kr.general.dichotomy}, and follow an argument similar to the proof of Theorem 5 of \cite{R08}.

This also allows us to extend Corollary \ref{cor:measurable.is.ramsey}:

\begin{corollary}
\label{cor:analytically.measurable.is.ramsey}
    Suppose that $(\cal{R},\leq,r)$ is a closed triple satisfying \textbf{A1}-\textbf{A4} and assume that $\cal{AR}$ is countable. If $\X \subseteq \cal{R}$ is in the smallest algebra of subsets of $\cal{R}$ containing all analytic sets, then $\X$ is Ramsey.
\end{corollary}

\begin{proof}
    The set of Ramsey subsets of $\cal{R}$ is closed under complements by definition. By Proposition \ref{prop:kr.closed.under.countable.unions}, the set of Ramsey subsets of $\cal{R}$ is closed under countable intersections, so it forms a $\sigma$-algebra. By Theorem \ref{thm:analytic.is.kr} and Theorem \ref{thm:kr.iff.r}, every analytic subset of $\cal{R}$ is contained in this $\sigma$-algebra.
\end{proof}

We remark that the abstract Rosendal theorem in \cite{d20} uses a similar approach to prove that every analytic subset of $X^\omega$ of a Gowers space is strategically Ramsey, where given a Gowers space, de Rancourt constructed a second Gowers space which equips a binary sequence along with elements of $X$.

\subsection{\texorpdfstring{$\Sigma_2^1$ Well Ordering}{Sigma\_2\^1 Well Ordering}}
We dedicate this section to showing that Theorem \ref{thm:analytic.is.kr} is consistently optimal for a family of sufficiently well-behaved \textbf{wA2}-space. 

\begin{definition}
    Let $(\cal{R},\leq,r)$ be a \textbf{wA2}-space. We say that $\cal{R}$ is \emph{deep} if for all $A \in \cal{R}$, $a \in \cal{AR}\restrictedto A$ and $N < \omega$, there exists some $B \in [a,A]$ such that for all $b \in r_{\lh(a)+1}[a,B]$, $\depth_A(b) \geq N$.
\end{definition}

We shall see later in the proof of Corollary \ref{cor:V=L.gives.non.kr.set.in.examples} that all examples of \textbf{wA2}-spaces introduced in \Sec2.2, except for the singleton space (Example \ref{ex:singleton.space}) are deep. The singleton space is not deep as deepness implies that for all $A \in \cal{R}$ and $a \in \cal{AR}\restrictedto[a,A]$, $r_{\lh(a)+1}[a,A]$ is infinite. We do not know if there are any ``natural'' examples of \textbf{wA2}-spaces which are not deep.

The main theorem of this section is as follows. 

\begin{theorem}
\label{thm:sigma_1^2.wo.gives.non.kr.set}
    Let $(\cal{R},\leq,r)$ be a deep \textbf{wA2}-space, and assume that $\cal{AR}$ is countable. Suppose that there exists a $\Sigma_2^1$-good well-ordering of the reals. Then there exists a $\mathbf{\Sigma}_2^1$ subset of $\cal{R}$ which is not Kastanas Ramsey.
\end{theorem}

See also Theorem IV.7.4 of \cite{AT05}. Observing that if $(\cal{R},\leq,r)$ is deep, then so is $(\cal{R} \times 2^\omega,\preceq,r)$, we obtain the following corollary:

\begin{corollary}
\label{cor:r.times.2^omega.has.coanalytic.non.kr.set}
    Let $(\cal{R},\leq,r)$ be a deep \textbf{wA2}-space, and assume that $\cal{AR}$ is countable. Then there exists a coanalytic subset of $\cal{R} \times 2^\omega$ which is not Kastanas Ramsey.
\end{corollary}

We now furnish a proof of Theorem \ref{thm:sigma_1^2.wo.gives.non.kr.set}. We shall now introduce two related games, which serve as a ``reduction'' of the Kastanas game for each player.

\begin{definition}
\label{def:fusion.game}
    Let $(\cal{R},\leq,r)$ be a \textbf{wA2}-space. Let $A \in \cal{R}$ and $a \in \cal{AR}\restrictedto A$. The \emph{fusion game} played below $[a,A]$, denoted as $Z[a,A]$, is defined as a game played by Player \textbf{I} and \textbf{II} in the following form:
    \begin{center}
        \begin{tabular}{c|c|c}
           Turn & \textbf{I} & \textbf{II} \\
            \hline
           1 & $A_0 \in [a,A]$ & \\
            & & $a_1 \in r_{\lh(a)+1}[a,A_0]$ \\
            \hline
           2 & $A_1 \in [a_1,A_0]$ & \\
            & & $a_2 \in r_{\lh(a_1)+1}[a_1,A_1]$ \\
            \hline
           3 &  $A_2 \in [a_2,A_1]$ & \\
            & & $a_3 \in r_{\lh(a_2)+1}[a_2,A_2]$ \\
            \hline
            $\vdots$ & $\vdots$ & $\vdots$
        \end{tabular}
    \end{center}
    The outcome of this game is $\lim_{n \to \infty} a_n \in \cal{R}$. We say that \textbf{I} (resp. \textbf{II}) has a strategy in $Z[a,A]$ to \emph{reach} $\X \subseteq \cal{R}$ if it has a strategy in $Z[a,A]$ to ensure that the outcome is in $\X$. 
\end{definition}

The following lemma, which roughly states that $Z[a,A]$ is a ``reduction'' of $K[a,A]$ for \textbf{I}, is obvious.

\begin{lemma}
    Let $(\cal{R},\leq,r)$ be a \textbf{wA2}-space. For any $A \in \cal{R}$ and $a \in \cal{AR}$, if \textbf{I} has a strategy in $K[a,A]$ to reach $\X$, then \textbf{I} has a strategy in $Z[a,A]$ to reach $\X$.
\end{lemma}

Similar to the game in Definition IV.7.2 of \cite{AT05}, it is possible to modify the fusion game to ensure that the set of all partial states is countable.

\begin{definition}
\label{def:countable.fusion.game}
    Let $(\cal{R},\leq,r)$ be a \textbf{wA2}-space. Let $A \in \cal{R}$ and $a \in \cal{AR}\restrictedto A$. The game $Z^*[a,A]$ is defined as a game played by Player \textbf{I} and \textbf{II} in the following form:
    \begin{enumerate}
        \item Player \textbf{I} begin by playing some $b_1^0 \in r_{\lh(a)+1}[a,A]$.

        \item Player \textbf{II} may choose to either respond with some $a_1 \in \{b \in r_{\lh(a)+1}[a,A] : b \leq_\fin b_1^0\}$, or not respond, in which case \textbf{I} plays some $b_1^1 \in r_{\lh(b_1^0)+1}[b_1^0,A]$.

        \item Repeat (2) until \textbf{II} chooses to respond with some $a_1 \in \{b \in r_{\lh(a)+1}[a,A] : b \leq_\fin b_1^k\}$ for some $k$. Then \textbf{I} responds by playing some $b_2^0 \in r_{\lh(a_1)+1}[a_1,A]$.

        \item Again, Player \textbf{II} may choose to either respond with some $a_2 \in \{b \in r_{\lh(a_1)+1}[a,A] : b \leq_\fin b_2^0\}$, or not respond, in which case \textbf{I} plays some $b_2^1 \in r_{\lh(b_2^0)+1}[b_2^0,A]$.

        \item Repeat.
    \end{enumerate}
    The outcome of this game is $\lim_{n \to \infty} a_n$. We say that \textbf{I} has a strategy in $Z^*[a,A]$ to \emph{reach} $\X \subseteq \cal{R}$ if either $\lim_{n \to \infty} a_n \notin \cal{R}$ (i.e. \textbf{II} stops playing after some finite stage), or \textbf{I} has a strategy in $Z[a,A]$ to ensure that the outcome is in $\X$. \textbf{II} has a strategy in $Z^*[a,A]$ to \emph{reach} $\X \subseteq \cal{R}$ if $\lim_{n \to \infty} a_n \in \X$.
\end{definition}

It is also easy to see that this gives us another reduction for \textbf{I}.

\begin{lemma}
    Let $(\cal{R},\leq,r)$ be a \textbf{wA2}-space. For any $A \in \cal{R}$ and $a \in \cal{AR}$, if \textbf{I} has a strategy in $Z[a,A]$ to reach $\X$, then \textbf{I} has a strategy in $Z^*[a,A]$ to reach $\X$.
\end{lemma}

We shall now introduce the reduction of $K[a,A]$ for \textbf{II}. 

\begin{definition}
\label{def:reverse.fusion.game}
    Let $(\cal{R},\leq,r)$ be a \textbf{wA2}-space. The \emph{subasymptotic game} played below $[a,A]$, denoted as $Y[a,A]$, is defined as a game played by Player \textbf{I} and \textbf{II} in the following form:
    \begin{center}
        \begin{tabular}{c|c|c}
           Turn & \textbf{I} & \textbf{II} \\
            \hline
           1 & $n_0 < \omega$ & \\
            & & $a_1 \in r_{\lh(a)+1}[a,A]$ s.t. $\depth_A(a_1) \geq n_0$ \\
            \hline
           2 & $n_1 < \omega$ & \\
            & & $a_2 \in r_{\lh(a_1)+1}[a_1,A]$ s.t. $\depth_A(a_2) \geq n_1$ \\
            \hline
           3 &  $n_2 < \omega$ & \\
            & & $a_3 \in r_{\lh(a_2)+1}[a_2,A]$ s.t. $\depth_A(a_3) \geq n_2$ \\
            \hline
            $\vdots$ & $\vdots$ & $\vdots$
        \end{tabular}
    \end{center}
    The outcome of this game is $\lim_{n \to \infty} a_n \in \cal{R}$. We say that \textbf{I} (resp. \textbf{II}) has a strategy in $Y[a,A]$ to \emph{reach} $\X \subseteq \cal{R}$ if it has a strategy in $Y[a,A]$ to ensure that the outcome is in $\X$. 
\end{definition}

\begin{lemma}
    Let $(\cal{R},\leq,r)$ be a deep \textbf{wA2}-space. For any $A \in \cal{R}$ and $a \in \cal{AR}$, if \textbf{II} has a strategy in $K[a,A]$ to reach $\X$, then \textbf{II} has a strategy in $Y[a,A]$ to reach $\X$.
\end{lemma}

\begin{proof}
    We first note that since $\cal{R}$ is deep, it is always possible for \textbf{II} to respond with a legal move in $Y[a,A]$. The deepness of $\cal{R}$ also allows us to view $Y[a,A]$ as a ``special case'' of $K[a,A]$: In $K[a,A]$, if \textbf{II} responded with $(a_k,B_{k-1})$, and \textbf{I} wants to restrict the next response by \textbf{II} such that $\depth_A(a_{k+1}) \geq n_k$ for some $n_k < \omega$, then \textbf{I} can respond to $(a_k,B_{k-1})$ by playing any $A_k \in [a_k,B_{k-1}]$ such that for all $b \in r_{\lh(a_k)+1}[a_k,A_k]$, $\depth_A(b) \geq n_k$. Therefore, a strategy for \textbf{II} in $K[a,A]$ to reach $\X$ may be passed to a strategy for \textbf{II} in $Y[a,A]$ to reach $\X$.
\end{proof}

\begin{definition}
    Let $(\cal{R},\leq,r)$ be a \textbf{wA2}-space. A set $\X \subseteq \cal{R}$ is \emph{pre-Kastanas Ramsey} if for all $A \in \cal{R}$ and $a \in \cal{AR}\restrictedto A$, there exists some $B \in [a,A]$ such that one of the following holds:
    \begin{enumerate}
        \item \textbf{I} has a strategy in $Z^*[a,B]$ to reach $\X^c$.

        \item \textbf{II} has a strategy in $Y[a,B]$ to reach $\X$.
    \end{enumerate}
\end{definition}

It is clear from the definition that every Kastanas Ramsey set is pre-Kastanas Ramsey. 

\begin{proof}[Proof of Theorem \ref{thm:sigma_1^2.wo.gives.non.kr.set}]
    It suffices to construct a $\mathbf{\Sigma}_2^1$ subset of $\cal{R}$ which is not pre-Kastanas Ramsey. Define the following two sets:
    \begin{align*}
        S_\textbf{I} &:= \{(a,A,\sigma) : a \in \cal{AR}\restrictedto A \wedge \text{$\sigma$ is a strategy for \textbf{I} in $Z^*[a,A]$}\}, \\
        S_\textbf{II} &:= \{(a,A,\sigma) : a \in \cal{AR}\restrictedto A \wedge \text{$\sigma$ is a strategy for \textbf{II} in $Y[a,A]$}\}.
    \end{align*}
    Note that a strategy in a game is a function from the set of partial states of the game to a play of the game. Since $\cal{AR}$ is countable, the sets of partial states of $Z^*[a,A]$ and of $Y[a,A]$ are also countable, and in both games all players play from a countable set. Therefore, we may naturally embed both sets $S_\textbf{I}$ and $S_\textbf{II}$ to the reals, giving us a $\Sigma_2^1$-well ordering $\prec_s$ of $S_\textbf{I}$ and $S_\textbf{II}$. Note that $\prec_s$ is of order-type $\omega_1$, so every triple in $S_\textbf{I} \cup S_\textbf{II}$ has countably many $\prec_s$-predecessors.
    
    Given a triple $(a,A,\sigma) \in S_\textbf{I} \cup S_\textbf{II}$, we build $B_{a,A,\sigma} \in \cal{R}$ by an increasing sequence $a = b_0 \sqsubseteq b_1 \sqsubseteq \cdots$ with $\lh(b_n) = \lh(a) + n$, and let $B_{a,A,\sigma} := \lim_{n\to\infty} b_n$. We consider two cases.
    \begin{enumerate}
        \item If $(a,A,\sigma) \in S_\textbf{I}$, then we shall construct some $B_{a,A,\sigma} \in \cal{R}$ such that $B_{a,A,\sigma}$ is the outcome of some full play in $Z^*[a,A]$, with \textbf{I} following $\sigma$, and that $B_{a,A,\sigma} \neq B_{a',A',\sigma'}$ for all $(a',A',\sigma') \prec_s (a,A,\sigma)$. We do this as follows: We first enumerate the $\prec_s$-predecessors by $\{(a_n,A_n,\sigma_n) : n < \omega\}$. Following a play in $Z^*[a,A]$ where \textbf{I} follows $\sigma$, suppose that \textbf{I} started the $n^\text{th}$ turn with $B_n \in [b_n,A]$. If $b_n \neq r_{\lh(a)+n}(B_{a_n,A_n,\sigma_n})$ or $B_{a_n,A_n,\sigma_n} \not\leq A$, then pick any $b_{n+1} \in r_{\lh(a)+n+1}[b_n,B_n]$. Otherwise, since $\cal{R}$ is deep we may pick some $b_{n+1} \in r_{\lh(a)+n+1}[b_n,B_n]$ such that $\depth_A(b_{n+1}) > \depth_A(r_{\lh(a)+n+1}(B_{a_n,A_n,\sigma_n}))$. This construction ensures that $B_{a,A,\sigma} \neq B_{a_n,A_n,\sigma_n}$ for all $n < \omega$.

        \item If $(a,A,\sigma) \in S_\textbf{II}$, then we shall construct some $B_{a,A,\sigma} \in \cal{R}$ such that $B_{a,A,\sigma}$ is the outcome of some full play in $Y[a,A]$, with \textbf{II} following $\sigma$, and that $B_{a,A,\sigma} \neq B_{a',A',\sigma'}$ for all $(a',A',\sigma') \prec_s (a,A,\sigma)$. We do this as follows: Again, we enumerate the $\prec_s$-predecessors by $\{(a_n,A_n,\sigma_n) : n < \omega\}$. Following a play in $Y[a,A]$ where \textbf{II} follows $\sigma$, suppose that the sequence $b_n$ has been played so far. If $B_{a_n,A_n,\sigma_n} \leq A$, we then ask that \textbf{I} respond with $\depth_A(r_{\lh(a)+n+1}(B_{a_n,A_n,\sigma_n})) + 1$, so that for any $b_{n+1}$ that \textbf{II} respond with next, $b_{n+1} \neq r_{\lh(a)+n+1}(B_{a_n,A_n,\sigma_n}))$. Otherwise, \textbf{I} may respond with any $k < \omega$. This construction ensures that $B_{a,A,\sigma} \neq B_{a_n,A_n,\sigma_n}$ for all $n < \omega$.
    \end{enumerate}
    Now let:
    \begin{align*}
        \X := \{B_{a,A,\sigma} : (a,A,\sigma) \in S_\textbf{I}\}.
    \end{align*}
    $\X$ is $\mathbf{\Sigma}_2^1$, as the well-ordering $\prec_s$ is $\Sigma_2^1$ and the construction of $\X$ is natural from $\prec_s$. We then see that $\X$ is not pre-Kastanas Ramsey: Let $A \in \cal{R}$ and $a \in \cal{AR}$.
    \begin{enumerate}
        \item If $\sigma$ is a strategy for \textbf{I} in $Z^*[a,A]$, then $B_{a,A,\sigma}$ is the outcome of a run following $\sigma$ such that $B_{a,A,\sigma} \notin \X^c$, so $\sigma$ is not a winning strategy for \textbf{I}.

        \item If $\sigma$ is a strategy for \textbf{II} in $Y[a,A]$, then $B_{a,A,\sigma}$ is the outcome of a run following $\sigma$ such that $B_{a,A,\sigma} \notin \X$, so $\sigma$ is not a winning strategy for \textbf{II}.
    \end{enumerate}
    This completes the proof.
\end{proof}

We remark that the set $\{B_{a,A,\sigma} : (a,A,\sigma) \in S_\textbf{I}\}$ would similarly produce a $\mathbf{\Pi}_2^1$ subset of $\cal{R}$ that is not pre-Kastanas Ramsey. Since such a well-ordering exists in G\"{o}del's constructible universe, we may conclude that:

\begin{corollary}[$\mathsf{V=L}$]
\label{cor:V=L.gives.non.kr.set}
    Let $(\cal{R},\leq,r)$ be a deep \textbf{wA2}-space, and assume that $\cal{AR}$ is countable. Then there exists a $\mathbf{\Sigma}_2^1$ subset of $\cal{R}$ which is not Kastanas Ramsey.
\end{corollary}

\begin{corollary}[$\mathsf{V=L}$]
\label{cor:V=L.gives.non.kr.set.in.examples}
    The following \textbf{wA2}-spaces have a $\mathbf{\Sigma}_2^1$ subset which is not Kastanas Ramsey:
    \begin{enumerate}
        \item $([\N]^\infty,\subseteq,r)$.

        \item $(\FIN_k^{[\infty]},\leq,r)$, the topological Ramsey space of infinite block sequences.
        
        \item $(\FIN_{\pm k}^{[\infty]},\leq,r)$, a variant of the space of infinite block sequences.

        \item $(W_{Lv}^{[\infty]},\leq,r)$, the Hales-Jewett space.

        \item $(\mathcal{S}^\infty,\subseteq,r)$, the topological Ramsey space of strong subtrees.

        \item $(\E^\infty,\leq,r)$, the Carlson-Simpson space.

        \item $(E^{[\infty]},\leq,r)$, the space of infinite-dimensional block subspaces of a countable vector space.
    \end{enumerate}
\end{corollary}

\begin{proof}
    By Corollary \ref{cor:V=L.gives.non.kr.set}, it suffices to show that every \textbf{wA2}-space above is deep.
    \begin{enumerate}
        \item Let $A = \{n_0,n_1,\dots\} \in [\N]^\infty$, and let $a \subseteq A$ be finite. For all $N$ such that $n_N > \max(a)$, we have that $\depth_A(a \cup \{n_{N-1}\}) = N$. Therefore, $a \cup \{n_{N-1}\} \in r_{\lh(a)+1}[a,A]$ and $\depth_A(a \cup \{n_{N-1}\}) \geq N$.
        
        \item Let $A = (x_0,x_1,\dots) \in \FIN_k^{[\infty]}$, and let $a \in \FIN_k^{[<\infty]}\restrictedto A$. For all $N$ such that $x_N > a$, we have that $\depth_A(a^\frown x_{N-1}) = N$. Therefore, $a^\frown x_{N-1} \in r_{\lh(a)+1}[a,A]$ and $\depth_A(a^\frown x_{N-1}) \geq N$.
        
        \item The proof is identical to that of $(\FIN_k^{[\infty]},\leq,r)$.
        
        \item Let $A = (x_0,x_1,\dots) \in W_{Lv}^{[\infty]}$, and let $a = (y_i)_{i<n} \in W_{Lv}^{[<\infty]}\restrictedto A$. Since $A$ is rapidly increasing, for $N$ large enough we have that $\sum_{i<n} |y_i| < |x_N|$. Therefore, $a^\frown x_N \in r_{\lh(a)+1}[a,A]$ and $\depth_A(a^\frown x_N) \geq N$.
        
        \item Let $A \subseteq T$ be a strong subtree. Given any $a \in \S_{<\infty}\restrictedto A$, we let $S_a$ be the set of terminal nodes in $a$. Since $a$ is a strong subtree, $S_a \subseteq \func{split}_N(A)$ for some $N$. We observe that:
        \begin{align*}
            \depth_A(a) = N \iff S_a \subseteq \func{split}_N(A).
        \end{align*}
        Fix any $a \in \S_{<\infty}\restrictedto A$ and $N < \omega$ be large enough. For each $s \in S_a$, let $t_{s,0},t_{s,1} \in \func{split}_N(A)$ be such that $s \sqsubseteq t_{s,0}$, $s \sqsubseteq t_{s,1}$ and $t_{s,0} \neq t_{s,1}$. We let:
        \begin{align*}
            b := \{u \in A : u \sqsubseteq t_{s,i} \text{ for some $s \in S_a$ and $i \in \{0,1\}$}\}.
        \end{align*}
        Observe that $a$ is an initial segment of $b$, and every terminal node in $a$ splits in $b$. Thus, $b \in r_{\lh(a)+1}[a,A]$ and $\depth_A(b) = N$.

        \item Let $A \in \E_\infty$, and let $a \in \cal{AE}_\infty\restrictedto A$. Let $\{p_n(A) : n < \omega\}$ be the increasing enumeration of the minimal representatives of $A$. Given $N > m := \depth_A(a)$, we define the equivalence relation $b$ on $\{0,1,\dots,p_N(A)-1\}$ as follows: Given $i,j \in \N$, we define:
        \begin{small}
            \begin{gather*}
                (i,j) \in b \iff
                \begin{cases}
                    (i,j) \in A \text{ and } (i \in \dom(a) \text{ or } j \in \dom(a)), \text{ or}; \\
                    (i \notin \dom(a) \text{ and } j \notin \dom(a)).
                \end{cases}
            \end{gather*}
        \end{small}
        We shall show that $b \in r_{\lh(a)+1}[a,A]$ and $\depth_A(b) = N$.
        
        Given $i,j < p_N(A)$ such that $(i,j) \in A$, if $i \in \dom(a)$ or $j \in \dom(a)$ then $(i,j) \in b$. Otherwise, $(i,j) \in b$ as well. Therefore, $b$ is an equivalence relation on $\dom(r_N(A))$ which is coarser than $A$, so $b \leq_\fin r_N(A)$. To see that $a \sqsubseteq b$ - if $i,j \in \dom(a)$ and $(i,j) \in b$, then $(i,j) \in A$, so $(i,j) \in a$ as $a$ is coarser than $A$. Finally, the equivalence classes in $b$ are either of the form $[i]_b$ for some $i \in \dom(a)$ (of which $(i,j) \notin a$ implies that $[i]_b \neq [j]_b$), or $[i]_b$ for any $i \notin \dom(a)$ (of which $[i]_b = [j]_b$ for all $i,j \notin \dom(a)$). Therefore, $b$ has $\lh(a) + 1$ many equivalence classes, i.e. $b \in r_{\lh(a)+1}[a,A]$.

        \item The proof is identical to that of $(\FIN_k^{[\infty]},\leq,r)$.
    \end{enumerate}
\end{proof}

\section{Strategically Ramsey Sets and Gowers Spaces}
\subsection{Gowers Spaces} 
de Rancourt first introduced Gowers spaces in \cite{d20} as a common abstraction to the topological Ramsey space $[\N]^\infty$ (i.e. the Ellentuck space, or Mathias-Silver space) and countable vector spaces (i.e. Rosendal space). We recall the definition.

\begin{definition}[Definition 2.1, \cite{d20}]
\label{def:gowers.space}
    A \emph{Gowers space} is a quintuple $(P,X,\leq,\leq^*,\lhd)$, where $P \neq \emptyset$ is the set of \emph{subspaces}, $X \neq \emptyset$ is at most countable (the set of \emph{points}), $\leq,\leq^*$ are two quasi-orders on $P$, and $\lhd \subseteq X^{<\omega} \times P$ is a binary relation, satisfying the following properties:
    \begin{enumerate}
        \item For all $p,q \in P$, if $p \leq q$, then $p \leq^* q$.

        \item For all $p,q \in P$, if $p \leq^* q$, then there exists some $r \in P$ such that $r \leq p$, $r \leq q$ and $p \leq^* r$.

        \item For every $\leq$-decreasing sequence $(p_n)_{n<\omega}$ of $P$, there exists some $p^* \in P$ such that $p^* \leq^* p_n$ for all $n < \omega$.

        \item For all $p \in P$ and $s \in X^{<\omega}$, there exists some $x \in X$ such that $s^\frown x \lhd p$.

        \item For all $s \in X^{<\omega}$ and $p,q \in P$, if $s \lhd p$ and $p \leq q$, then $s \lhd q$.
    \end{enumerate}
    Given $p,q \in P$, we also write $p \lessapprox q$ iff $p \leq q$ and $q \leq^* p$.
\end{definition}

de Rancourt proceeded to introduce various games in this abstract setting. We hereby provide a summary of the games we're interested in. Note that we have employed some changes in the names/notations of the game.

\begin{definition}[Definition 2.2, \cite{d20}]
    For each $p \in P$, the \emph{adversarial Gowers game} $AG(p)$ is defined as a game played by Player \textbf{I} and \textbf{II} in the following form:
    \begin{center}
        \begin{tabular}{c|cccccc}
            \textbf{I} & & $x_0,q_0$ & & $x_1,q_1$ & & $\cdots$ \\
            \hline
            \textbf{II} & $p_0$ & & $y_0,p_1$ & & $y_1,p_2$ & $\cdots$ \\
        \end{tabular}
    \end{center}
    such that $x_n,y_n \in X$ and $p_n,q_n \in P$ for all $n$, and that the following additional condition must be fulfilled for all $n < \omega$:
    \begin{enumerate}
        \item $(x_0,y_0,\dots,x_{n-1},y_{n-1},x_n) \lhd p_n$.

        \item $(x_0,y_0,\dots,x_n,y_n) \lhd q_n$.

        \item $p_n \leq p$ and $q_n \leq p$.
    \end{enumerate}
    The outcome of this game is $(x_0,y_0,x_1,y_1,\dots)$. We say that \textbf{I} (resp. \textbf{II}) has a strategy in $K(p)$ to \emph{reach} $\X \subseteq X^\omega$ if it has a strategy in $K(p)$ to ensure that the outcome is in $\X$. 
\end{definition}

\begin{definition}[Definition 2.2, \cite{d20}]
    For each $p \in P$, the \emph{adversarial Gowers game for \textbf{I}} $AG_\textbf{I}(p)$ (resp. \emph{for \textbf{II}} $AG_\textbf{II}(p)$) is the game $AG(p)$ with the following additional restrictions:
    \begin{enumerate}
        \item For $AG_\textbf{I}(p)$, \textbf{I} can only play $q_n$ such that $q_n \lessapprox p$.

        \item For $AG_\textbf{I}(p)$, \textbf{II} can only play $p_n$ such that $p_n \lessapprox p$.
    \end{enumerate}
\end{definition}

\begin{definition}[Definition 2.5, \cite{d20}]
    For each $p \in P$, the \emph{de Rancourt game} $R(p)$ is the game $AG(p)$ with the following additional restriction:
    \begin{enumerate}
        \item For all $n < \omega$, $q_n \leq p_n$ and $p_{n+1} \leq q_n$.
    \end{enumerate}
\end{definition}

\begin{definition}[Definition 3.1, \cite{d20}]
    For each $p \in P$, the \emph{Gowers game} $G(p)$ is defined as a game played by Player \textbf{I} and \textbf{II} in the following form:
    \begin{center}
        \begin{tabular}{c|cccccc}
            \textbf{I} & $p_0$ & & $p_1$ & & $\cdots$ \\
            \hline
            \textbf{II} & & $x_0$ & & $x_1$ & $\cdots$ \\
        \end{tabular}
    \end{center}
    such that $x_n \in X$ and $p_n \in P$ for all $n$, and that the following additional condition must be fulfilled for all $n < \omega$:
    \begin{enumerate}
        \item $(x_0,\dots,x_n) \lhd p_n$.

        \item $p_n \leq p$.
    \end{enumerate}
    The outcome of this game is $(x_0,x_1,\dots)$. We say that \textbf{I} (resp. \textbf{II}) has a strategy in $K(p)$ to \emph{reach} $\X \subseteq X^\omega$ if it has a strategy in $K(p)$ to ensure that the outcome is in $\X$. 
\end{definition}

\begin{definition}[Definition 3.1, \cite{d20}]
    For each $p \in P$, the \emph{asymptotic game} $F(p)$ is the game $G(p)$ with the following additional restriction:
    \begin{enumerate}
        \item For all $n < \omega$, $p_n \lessapprox p$.
    \end{enumerate}
\end{definition}

We now introduce several variants of game-theoretic Ramsey properties.

\begin{definition}[Definition 2.3, \cite{d20}]
    A set $\X \subseteq X^\omega$ is \emph{adversarially Ramsey} if for all $p \in P$, there exists some $q \leq p$ such that one of the following holds:
    \begin{enumerate}
        \item \textbf{I} has a strategy in $AG_\textbf{I}(q)$ to reach $\X$.

        \item \textbf{II} has a strategy in $AG_\textbf{II}(q)$ to reach $\X^c$.
    \end{enumerate}
\end{definition}

\begin{definition}[Definition 3.2, \cite{d20}]
    A set $\X \subseteq X^\omega$ is \emph{strategically Ramsey} if for all $p \in P$, there exists some $q \leq p$ such that one of the following holds:
    \begin{enumerate}
        \item \textbf{I} has a strategy in $F(q)$ to reach $\X^c$.

        \item \textbf{II} has a strategy in $G(q)$ to reach $\X$.
    \end{enumerate}
\end{definition}

\begin{definition}
    A set $\X \subseteq X^\omega$ is \emph{de Rancourt Ramsey} if for all $p \in P$, there exists some $q \leq p$ such that one of the following holds:
    \begin{enumerate}
        \item \textbf{I} has a strategy in $R(q)$ to reach $\X$.

        \item \textbf{II} has a strategy in $R(q)$ to reach $\X^c$.
    \end{enumerate}
\end{definition}

\begin{proposition}
\label{prop:rr.iff.ar}
    Let $p \in P$ and $\X \subseteq X^\omega$.
    \begin{enumerate}
        \item \textbf{I} has a strategy in $R(q)$ to reach $\X$ for some $q \leq p$ iff there exists some $q \leq p$ such that \textbf{I} has a strategy in $AG_\mathbf{I}(q)$ to reach $\X$.

        \item \textbf{II} has a strategy in $R(q)$ to reach $\X$ for some $q \leq p$ iff there exists some $q \leq p$ such that \textbf{II} has a strategy in $AG_\mathbf{II}(q)$ to reach $\X$.
    \end{enumerate}
\end{proposition}

\begin{proof}
    The forward direction for both statements has been proven in Proposition 2.6 of \cite{d20}, so we only prove the converse for (1) (the proof for the converse for (2) is almost verbatim). Suppose $\sigma$ is a strategy for \textbf{I} in $AG_\textbf{I}(p)$ to reach $\X$, and we define a strategy $\tau$ for \textbf{I} in $R(p)$. For each state $s$ for \textbf{II} of $R(p)$ following $\tau$, we shall correspond it to a state $t_s$ for \textbf{II} of $AG_\textbf{I}(p)$ realising $a(s)$. Start by letting $t_{(p_0)} := (p_0)$ for any $p_0 \leq p$. Now let $s$ be a state of $R(p)$ for \textbf{II} following $\tau$ so far, with $s = {s'}^\frown(y_n,p_{n+1})$. Suppose by the induction hypothesis that there exists a corresponding state $t_s$ of the game $AG_\textbf{I}(p)$ such that:
    \begin{enumerate}
        \item $a(s') = a(t_{s'})$;

        \item $\sigma(t_{s'}) = (x_n,q_n')$ for some $q_n' \lessapprox p$.
    \end{enumerate}
    Now let $t_s := {t_{s'}}^\frown(y_n,p_{n+1})$, and suppose $\sigma(t_s) = (x_{n+1},q_{n+1}')$ for some $q_{n+1}' \lessapprox p$. Since $p_{n+1} \leq p$, by Property (2) of Definition \ref{def:gowers.space} there exists some $q_{n+1} \leq p_{n+1}$ such that $q_{n+1} \leq q_{n+1}'$. Then $\tau(s) := (x_{n+1},p_{n+1})$ is a legal continuation. This completes the inductive definition of $\tau$, which is a winning strategy as every complete play following $\tau$ corresponds to a complete play following $\sigma$ realising the same sequence.
\end{proof}

\subsection{The Kastanas Game}
We now introduce (our version of) the Kastanas game for Gowers spaces.

\begin{definition}[Definition 2.5, \cite{d20}]
    For each $p \in P$, the \emph{Kastanas game} $K(p)$ is defined as a game played by Player \textbf{I} and \textbf{II} in the following form:
    \begin{center}
        \begin{tabular}{c|cccccc}
            \textbf{I} & $p_0$ & & $p_1$ & & $p_2$ & $\cdots$ \\
            \hline
            \textbf{II} & & $x_0,q_0$ & & $x_1,q_1$ & & $\cdots$ 
        \end{tabular}
    \end{center}
    such that $x_n \in X$ and $p_n,q_n \in P$ for all $n$, and that the following additional condition must be fulfilled for all $n < \omega$:
    \begin{enumerate}
        \item $(x_0,\dots,x_n) \lhd p_n$.

        \item $q_n \leq p_n$ and $p_{n+1} \leq q_n$.
    \end{enumerate}
    The outcome of this game is $(x_0,x_1,\dots)$. We say that \textbf{I} (resp. \textbf{II}) has a strategy in $K(p)$ to \emph{reach} $\X \subseteq X^\omega$ if it has a strategy in $K(p)$ to ensure that the outcome is in $\X$. 
\end{definition}

\begin{definition}
\label{def:kr.gowers}
    A set $\X \subseteq \X^\omega$ is \emph{Kastanas Ramsey} if for all $p \in P$, there exists some $q \leq p$ such that one of the following holds:
    \begin{enumerate}
        \item \textbf{I} has a strategy in $K(q)$ to reach $\X^c$.

        \item \textbf{II} has a strategy in $K(q)$ to reach $\X$.
    \end{enumerate}
\end{definition}

\begin{proposition}
\label{prop:kr.iff.sr}
    A subset $\X \subseteq X^\omega$ is Kastanas Ramsey iff $\X$ is strategically Ramsey (Definition 3.2, \cite{d20}).
\end{proposition}

We shall prove this proposition as a corollary of Proposition \ref{prop:rr.iff.ar}. 

\begin{proof}
    We let $(P,X,\leq,\leq^*,\lhd)$ be a Gowers space, and assume WLOG that $0 \notin X$. We then define a relation $\blhd \; \subseteq (X \cup \{0\})^{<\omega} \times P$ such that for all $s \in (X \cup \{0\})^{<\omega}$:
    \begin{enumerate}
        \item If $\lh(s)$ is odd, then for all $x \in X \cup \{0\}$ and $p \in P$, $s^\frown x \blhd P$ iff $x = 0$.
    
        \item If $\lh(s)$ is even, then for all $x \in X \cup \{0\}$ and $p \in P$, $s^\frown x \blhd P$ iff $x \neq 0$ and $s^\frown x \lhd P$.
    \end{enumerate}
    
    It is easy to verify that $(P,X \cup \{0\},\leq,\leq^*,\blhd)$ is a Gowers space. We now define an injective function $f : X^{<\omega} \to (X \cup \{0\})^{<\omega}$ by:
    \begin{align*}
        f((x_0,\dots,x_{n-1})) := (x_0,0,x_1,0\dots,x_{n-1},0)
    \end{align*}
    and naturally extend $f$ to $X^\omega \to (X \cup \{0\})^\omega$. Note that $f$ is injective. For each $p \in P$, we also define the functions $g,h$ by:
    \begin{align*}
        g(p_0,0,p_1,0,p_2,\dots) &:= (p_0,p_1,p_2,\dots), \\
        h(x_0,q_0,x_1,q_1,\dots) &:= (x_0,x_1,\dots).
    \end{align*}
    We may now observe that:
    \begin{enumerate}
        \item $\sigma$ is a strategy for \textbf{I} in $K(p)$ to reach $\X$ iff:
        \begin{align*}
            s \mapsto 
            \begin{cases}
                \sigma(s), &\text{if $s = \emptyset$}, \\
                (0,\sigma(s)), &\text{if $s \neq \emptyset$}. \\
            \end{cases}
        \end{align*}
        is a strategy for \textbf{II} in $R(p)$ to reach $f[\X]$.

        \item $\sigma$ is a strategy for \textbf{II} in $K(p)$ to reach $\X$ iff $\sigma \circ g$ is a strategy for \textbf{I} in $R(p)$ to reach $f[\X]$.

        \item $\sigma$ is a strategy for \textbf{I} in $F(p)$ to reach $\X$ iff:
        \begin{align*}
            s \mapsto 
            \begin{cases}
                (\sigma \circ h)(s), &\text{if $s = \emptyset$}, \\
                (0,(\sigma \circ h)(s)), &\text{if $s \neq \emptyset$}. \\
            \end{cases}
        \end{align*}
        is a strategy for \textbf{II} in $AG_\textbf{II}(p)$ to reach $f[\X]$.

        \item $\sigma$ is a strategy for \textbf{II} in $G(p)$ to reach $\X$ iff $s \mapsto ((\sigma \circ g)(s),p)$ is a strategy for \textbf{I} in $AG_\textbf{I}(p)$ to reach $f[\X]$.
    \end{enumerate}
    Therefore, the proposition follows from Proposition \ref{prop:rr.iff.ar}.
\end{proof}

\subsection{Gowers \textbf{wA2}-Spaces}
We shall now reformulate the above result in the context of \textbf{wA2}-spaces. In \cite{M07}, Mijares introduced the notion of an \emph{almost reduction} for spaces satisfying \textbf{A1}-\textbf{A4}, which may be applied to \textbf{wA2}-spaces. We introduce a variant of this almost reduction, restricted to a fixed initial segment.

\begin{notation}
    Let $(\cal{R},\leq,r)$ be a \textbf{wA2}-space. Given $A,B \in \cal{R}$ and $a \in \cal{AR}$, we write $A \leq_a^* B$ iff there exists some $b \in \cal{AR}\restrictedto[a,A]$ such that $[b,A] \subseteq [b,B]$.
\end{notation}

Note that $\leq_a^*$ need not be a transitive relation - counterexample would be the topological Ramsey space of strong subtrees (which satisfies \textbf{A1}-\textbf{A4}). Note also that, by \textbf{A1}, we may identify each element $a \in \cal{AR}$ with the sequence $(r_n(a))_{1 \leq n \leq \lh(a)} \in \cal{AR}^{<\omega}$.

\begin{definition}
    Let $(\cal{R},\leq,r)$ be a \textbf{wA2}-space. We say that $\cal{R}$ is \emph{Gowers} if there exists a relation $\lhd \subseteq \cal{AR} \times \cal{R}$ such that  the following properties hold:
    \begin{enumerate}
        \item[(\textbf{G1}-\textbf{5})] For all $A \in \cal{R}$ and $a \in \cal{AR}\restrictedto A$, $([a,A],\cal{AR}\restrictedto[a,A],\leq,\leq_a^*,\lhd)$ is a Gowers space (when identifying elements of $\cal{AR}$ with $\cal{AR}^{<\omega}$).

        \item[(\textbf{G6})] Let $A,B \in \cal{R}$. Let $a \in \cal{AR}\restrictedto A \cap \cal{AR}\restrictedto B$.
        \begin{enumerate}[label=(\arabic*)]
            \item $[a,A] \subseteq [a,B]$ iff $r_n(A) \lhd B$ for all $n > \depth_A(a)$.

            \item If there exists some $N$ such that $r_n(A) \lhd B$ for all $n \geq N$, then $A \leq_a^* B$.
        \end{enumerate}

        \item[(\textbf{G7})] For all $A \in \cal{R}$, $a \in \cal{AR}\restrictedto A$ and $B \leq A$, there exists some $C \in [\depth_A(a),A]$ such that for all $b \in \cal{AR}\restrictedto[a,A]$, if $b \lhd C$ then $b \lhd B$.
    \end{enumerate}
\end{definition}

\begin{example}[Natural numbers/Ellentuck space $\lbrack \N\rbrack^\infty$]
    We show that $([\N]^\infty,\subseteq,r)$ is a Gowers \textbf{wA2}-space.
    \begin{enumerate}
        \item[(\textbf{G1}-\textbf{5})] Let $A \in [\N]^\infty$ and $a \in [\N]^{<\infty}\restrictedto A$. Note for all $B,C \in [a,A]$, $C \leq_a^* B$ iff $C \setminus N \subseteq B$ for some $N \geq \max(a)$. Given $b = a \cup \{x_{|a|},\dots,x_n\} \in [\N]^{<\infty}\restrictedto[a,A]$ and $B \in [a,A]$, we define $b \lhd B$ iff $x_n \in B$.
        \begin{enumerate}[label=(\arabic*)]
            \item Clearly $C \subseteq B$ implies that $C \leq_a^* B$.
            
            \item If $C \leq_a^* B$, then there exists some $n$ such that $D := a \cup (C \setminus n) \subseteq B$. Then $D \subseteq C$, $D \subseteq B$ and $C \leq_a^* D$ as $(C \setminus a) \setminus (D \setminus a) \subseteq n$.
            
            \item Let $(B_n)_{n<\omega}$ be a $\subseteq$-decreasing sequence in $[a,A]$, and let $C \subseteq B_0 \setminus a$ be such that $C \subseteq^* B_n$ for all $n$. Then $a \cup C \in [a,A]$ and $a \cup C \leq_a^* B_n$ for all $n$.
            
            \item Given $B \in [a,A]$ and $b \in [\N]^{<\infty}\restrictedto[a,A]$, $b \cup \{x\} \lhd B$ for any $x \in B$ such that $\max(b) < x$.
            
            \item If $b = a \cup \{x_0,\dots,x_n\} \in [\N]^{<\infty}\restrictedto[a,A]$, $b \lhd C$ and $C \subseteq B$, then $x_n \in C \subseteq B$, so $b \lhd B$.
        \end{enumerate}

        \item[(\textbf{G6})] Let $A,B \in [\N]^\infty$ and $b \in [\N]^{<\infty}\restrictedto[a,A] \cap [\N]^{<\infty}\restrictedto[a,B]$. We write $A = \{x_0,x_1,\dots\}$ and $m := \depth_A(a)$ (i.e. $\max(a) = x_{m-1}$). 
        \begin{enumerate}[label=(\arabic*)]
            \item We have that:
            \begin{align*}
                &\iffbreak [a,A] \subseteq [a,B] \\
                &\iff A \setminus \max(a) \subseteq B \\
                &\iff x_n \in B \text{ for all $n \geq m$} \\
                &\iff (x_0,\dots,x_n) \lhd B \text{ for all $n \geq \depth_A(a)$} \\
                &\iff r_n(A) \lhd B \text{ for all $n > \depth_A(a)$}.
            \end{align*}

            \item If $\{x_0,\dots,x_n\} \lhd B$ for all $n \geq N > m$, then we have that $A \setminus (a \cup \{x_{|a|},\dots,x_N\}) \subseteq B$, so $A \leq_a^* B$.
        \end{enumerate}

        \item[(\textbf{G7})] Let $A \in [\N]^\infty$, $a \in [\N]^{<\infty}\restrictedto A$ and $B \subseteq A$. Let $m := \depth_A(a)$, and let $C := r_m(A) \cup (B \setminus \max(a))$. Then for all $b \in [\N]^{<\infty}[a,A]$, if $b = a \cup \{x_{|a|},\dots,x_n\}$ and $b \lhd C$, then $x_n > \max(a) = \max(r_m(A))$ and $x_n \in C \subseteq B$, so $c \lhd B$.
    \end{enumerate}
\end{example}

\begin{example}[Countable vector space $E^{[\infty]}$]
    We show that $(E^{[\infty]},\leq,r)$ is a Gowers \textbf{wA2}-space. Given some $A = (x_n)_{n<\omega} \in E^{[\infty]}$, we denote $A/N := (x_n)_{n \geq N}$.
    \begin{enumerate}
        \item[(\textbf{G1}-\textbf{5})] Let $A \in E^{[\infty]}$ and $a \in E^{[<\infty]}\restrictedto A$. Note for all $B,C \in [a,A]$, $C \leq_a^* B$ iff $C/N \leq B$ for some $N \geq \lh(a)$. Given $b = a^\frown(x_{|a|},\dots,x_n) \in E^{[<\infty]}\restrictedto[a,A]$ and $B \in [a,A]$, we define $b \lhd B$ iff $x_n \in \c{B}$.
        \begin{enumerate}[label=(\arabic*)]
            \item Clearly $C \leq B$ implies that $C \leq_a^* B$.
            
            \item If $C \leq_a^* B$, then there exists some $N \geq \lh(a)$ such that $D := a^\frown(C/N) \leq B$. Then $D \leq C$, $D \leq B$ and $C \leq_a^* D$ as $C/(\lh(a) + N) \leq D$.
            
            \item Let $(B_n)_{n<\omega}$ be a $\leq$-decreasing sequence in $[a,A]$, and let $C \leq B_0/\lh(a)$ be such that $C \leq^* B_n/\lh(a)$ for all $n$ ($C = (x_n)_{n<\omega}$ may be constructed by picking $x_n \in \c{B_n}$). Then $a^\frown C \in [a,A]$ and $a^\frown C \leq_a^* B_n$ for all $n$.
            
            \item Given $B \in [a,A]$ and $b \in E^{[<\infty]}\restrictedto[a,A]$, $b^\frown x \lhd B$ for any $x \in B$ such that $\max(b) < x$.
            
            \item If $b = a^\frown(x_{|a|},\dots,x_n) \in E^{[<\infty]}$, $b \lhd C$ and $C \leq B$, then $x_n \in \c{C} \subseteq \c{B}$, so $b \lhd B$.
        \end{enumerate}

        \item[(\textbf{G6})] Let $A,B \in E^{[\infty]}$ and $b \in E^{[\infty]}\restrictedto[a,A] \cap E^{[\infty]}\restrictedto[a,B]$. We write $A = (x_0,x_1,\dots)$ and $m := \depth_A(a)$. 
        \begin{enumerate}[label=(\arabic*)]
            \item We have that:
            \begin{align*}
                &\iffbreak [a,A] \subseteq [a,B] \\
                &\iff A/m \leq B \\
                &\iff x_n \in \c{B} \text{ for all $n \geq m$} \\
                &\iff (x_0,\dots,x_n) \lhd B \text{ for all $n \geq m$} \\
                &\iff r_n(A) \lhd B \text{ for all $n > \depth_A(a)$}.
            \end{align*}

            \item If $(x_0,\dots,x_n) \lhd B$ for all $n \geq N > m$, then we have that $A/N \leq B$, so $A \leq_a^* B$.
        \end{enumerate}

        \item[(\textbf{G7})] Let $A \in E^{[\infty]}$, $a \in E^{[<\infty]}\restrictedto A$ and $B \subseteq A$. Let $m := \depth_A(a)$, and let $C := r_m(A)^\frown(B/N)$, where $N > \max(\supp(a))$. Then for all $b \in E^{[\infty]}[a,A]$, if $b = a \cup \{x_{|a|},\dots,x_n\}$ and $b \lhd C$, then $\min(\supp(x_n)) > \max(\supp(a)) = \max(\supp(r_m(A)))$ and $x_n \in \c{C} \subseteq \c{B}$, so $c \lhd B$.
    \end{enumerate}
\end{example}

\begin{theorem}
\label{thm:kr.strategies.equivalence}
    Let $(\cal{R},\leq,r)$ be a Gowers \textbf{wA2}-space, and let $\X \subseteq \cal{R}$. Let $A \in \cal{R}$ and $a \in \cal{AR}$. The following are equivalent:
    \begin{enumerate}
        \item \textbf{I} (resp. \textbf{II}) has a strategy in $K[a,A]$ to reach $\X$.

        \item \textbf{I} (resp. \textbf{II}) has a strategy in $K(A)$ (as a game of the Gowers space $([a,A],\cal{AR}\restrictedto[a,A],\leq,\leq_a^*,\lhd)$) to reach $\X \cap [a,A]$.
    \end{enumerate}
\end{theorem}

\begin{proof}
    \underline{(1)$\implies$(2), Player \textbf{I}:} Let $\sigma$ be a strategy for \textbf{I} in $K[a,A]$ to reach $\X$. We define a strategy $\tau$ for \textbf{I} in $K(A)$ as follows: Let $s$ be a state for \textbf{II} in $K(A)$ following $\tau$ so far, and suppose that $s = {s'}^\frown(a_n,B_{n-1})$, and we have defined a state $t_{s'}$ for \textbf{II} in $K[a,A]$ such that $a(s) = a(t_{s'})$ (i.e. they realise the same finite sequence so far), and $\tau(s') \leq \sigma(t_{s'})$. Note that for the base case, we define $t_\emptyset := \emptyset$, and $\tau(\emptyset) := \sigma(\emptyset)$. 
    
    Since $B_n \leq \tau(s') \leq \sigma(t_{s'})$ and $a_n \sqsubseteq \sigma(t_{s'})$, by \textbf{G7} there exists some $C_{n-1} \in [a_n,\sigma(t_{s'})]$ such that for all $b \in \cal{AR}\restrictedto[a_n,C_n]$, if $b \lhd C_{n-1}$ then $b \lhd B_{n-1}$. We may thus define the legal continuation $t_s := {t_{s'}}^\frown(a_n,C_{n-1})$. Then by \textbf{G6}, $\sigma(t_s) \leq C_{n-1} \leq_a^* B_{n-1}$, so by \textbf{G2} (i.e. Property (2) of Definition \ref{def:gowers.space}) we may define $\tau(s) \leq B_n$ to be such that $\tau(s) \leq \sigma(t_s)$. This completes the inductive definition of $\tau$, and it is a winning strategy for \textbf{I} as every complete play $s$ of $K(A)$ following $\tau$ induces a complete play $t_s$ of $K[a,A]$ following $\sigma$, with the same outcome.

     \underline{(1)$\implies$(2), Player \textbf{II}:} Let $\sigma$ be a strategy for \textbf{II} in $K[a,A]$ to reach $\X$. We define a strategy $\tau$ for \textbf{II} in $K(A)$ as follows: Let $s$ be a state for \textbf{I} in $K(A)$ following $\tau$ so far, and suppose that $s = {s'}^\frown(\tau(s'),A_n)$, and we have defined a state $t_{s'}$ for \textbf{I} in $K[a,A]$ such that $a(s) = a(t_{s'})$, and $B_{n-1} \leq C_{n-1}$. Note that for the base case, we define $t_{(A_0)} := (A_0)$, and $\tau((A_0)) := \sigma((A_0))$.

     We write $\tau(s') = (a_n,B_{n-1})$ and $\sigma(t_{s'}) = (a_n,C_{n-1})$. Since $A_n \leq B_{n-1} \leq C_{n-1}$, by \textbf{G7} there exists some $A_n' \in [a_n,B_{n-1}]$ such that for all $b \in \cal{AR}\restrictedto[a_n,B_{n-1}]$, if $b \lhd A_n'$ then $b \lhd A_n$. We may thus define the legal continuation $t_s := {t_{s'}}^\frown(A_n')$. By \textbf{G7}, if $\sigma(t_{s'}) = (a_{n+1},C_n)$, then we may define $\tau(s) = (a_{n+1},B_n)$, where $B_n \leq A_n$ and $B_n \leq C_n$. This completes the inductive definition of $\tau$, and it is a winning strategy for \textbf{II} as every complete play $s$ of $K(A)$ following $\tau$ induces a complete play $t_s$ of $K[a,A]$ following $\sigma$, with the same outcome.

     The proof for (2)$\implies$(1) for both players is similar but simpler, mostly using \textbf{G7} to make the necessary changes to the strategy. For instance, suppose that $\sigma$ is a strategy for \textbf{I} in $K(A)$ to reach $\X \cap [a,A]$. Suppose that in the game $K[a,A]$, \textbf{II} responded with $(a_n,B_{n-1})$, and $\sigma$ then responds with some $A_n' \leq B_{n-1}$. By \textbf{G7}, we instead ask \textbf{I} to respond with $A_n \in [a_n,B_{n-1}]$ such that for all $a_{n+1} \in r_{\lh(a_n)+1}[a_n,B_{n-1}]$, if $a_{n+1} \lhd A_n$ then $a_{n+1} \lhd A_n'$. Then this modification gives \textbf{I} a strategy in $K[a,A]$ to reach $\X$.
\end{proof}

Alternatively, one may define the corresponding de Rancourt game for \textbf{wA2}-spaces, then define a corresponding notion of de Rancourt Ramsey, and prove using similar methods that the corresponding notion of de Rancourt Ramsey is equivalent to that for Gowers spaces. Then Theorem \ref{thm:kr.strategies.equivalence} may be deduced using the maps $g,h$ defined in the proof of Proposition \ref{prop:kr.iff.sr}.

Consequently, we have the following immediate corollary.

\begin{corollary}
\label{cor:kr.equivalence}
    Let $(\cal{R},\leq,r)$ be a Gowers \textbf{wA2}-space, and let $\X \subseteq \cal{R}$. The following are equivalent:
    \begin{enumerate}
        \item $\X$ is Kastanas Ramsey (as in Definition \ref{def:kastanas.ramsey}).

        \item For all $A \in \cal{R}$ and $a \in \cal{AR}$, $\X \cap [a,A]$ is a Kastanas Ramsey subset of $[a,A]$ (as in Definition \ref{def:kastanas.ramsey}).

        \item For all $A \in \cal{R}$ and $a \in \cal{AR}$, $\X \cap [a,A]$ is a Kastanas Ramsey subset of $[a,A]$ (as in Definition \ref{def:kr.gowers} for the Gowers space $([a,A],\cal{AR}\restrictedto[a,A],\leq,\leq_a^*,\lhd)$).
    \end{enumerate}
\end{corollary}

Since a countable vector space $E^{[\infty]}$ is a Gowers \textbf{wA2}-space with countable $\cal{AR}$, we may conclude all the following classical facts of strategically Ramsey sets:

\begin{corollary}
\label{cor:sr.properties}
    Let $(E^{[\infty]},\leq,r)$ be the \textbf{wA2}-space of infinite-dimensional block subspaces of a countable vector space.
    \begin{enumerate}
        \item\label{cor:kr.iff.sr} $\X \subseteq E^{[\infty]}$ is Kastanas Ramsey (as in Definition \ref{def:kastanas.ramsey}) iff $\X$ is strategically Ramsey (as in Definition 1 of \cite{R08}).

        \item Every analytic subset of $E^{[\infty]}$ is strategically Ramsey.

        \item The set of strategically Ramsey subsets of $E^{[\infty]}$ is closed under countable unions, but not under complement and finite intersection.
    \end{enumerate}
\end{corollary}

\begin{proof}
    \hfill
    \begin{enumerate}
        \item Combine Proposition \ref{prop:kr.iff.sr} and Corollary \ref{cor:kr.equivalence}.

        \item Combine Proposition \ref{prop:kr.iff.sr} and Theorem \ref{thm:analytic.is.kr}.

        \item By Example \ref{ex:biasymptotic.set.vector.spaces}, there is a biasymptotic subset of $E^{[<\infty]}$.
    \end{enumerate}
\end{proof}

\subsection{Coanalytic Sets}
Combined with Proposition \ref{prop:kr.iff.sr} and Theorem IV.7.5 of \cite{AT05}, we get a positive answer to Question \ref{qn:analytic.non.kr.set} in the context of a countable vector space. This section shows that this is, in fact, a consequence of Corollary \ref{cor:r.times.2^omega.has.coanalytic.non.kr.set} and a suitable choice of coding. 

Let $E$ be a vector space over a countable field with a dedicated Schauder basis $(e_n)_{n<\omega}$. Let $Y \subseteq E$ be a biasymptotic set (i.e. for all $A \in E^{[\infty]}$, $\c{A} \cap Y \neq \emptyset$ and $\c{A} \cap Y^c \neq \emptyset$). We first define $\delta : E \to 2$ by stipulating that:
\begin{align*}
    \delta(x) := 
    \begin{cases}
        1, &\text{if $x \in Y$}, \\
        0, &\text{if $x \notin Y$}.
    \end{cases}
\end{align*}
Fix some $a \in E^{[<\infty]}$ and $p \in 2^{\lh(a)}$. Let $E^{[<\infty]}(a)$ denote the set of all $b \in E^{[\infty]}$ such that $a \sqsubseteq b$. We now define a map $\Lambda_{a,p} : E^{[<\infty]}(a) \to \cal{A}(E^{[\infty]} \times 2^\omega)$ by stipulating that:
\begin{align*}
    \Lambda_{a,p}(a^\frown(x_k)_{k<n}) := (a^\frown(x_{2k})_{2k<n},p^\frown(\delta(x_{2k+1}))_{2k+1<\omega}),
\end{align*}
We may then extend this function to a continuous map $\Lambda_{a,p} : E^{[\infty]}(a) \to E^{[\infty]} \times 2^\omega$. By Corollary \ref{cor:sr.properties}(\ref{cor:kr.iff.sr}), we may replace Kastanas Ramsey sets with strategically Ramsey sets in our discussion.

\begin{lemma}
\label{lem:coding.preserves.F}
    Let $\cal{C} \subseteq E^{[\infty]} \times 2^\omega$. Let $A \in E^{[\infty]}$, $a \in E^{[<\infty]}\restrictedto A$ and $p \in 2^{\lh(a)}$. Suppose that \textbf{I} has a strategy in $F[a,A]$ to reach $\Lambda_{a,p}^{-1}[\cal{C}]$. Then \textbf{I} has a strategy in $F[(a,p),(A,\vec{0})]$ to reach $\cal{C}$.
\end{lemma}

\begin{proof}
    Let $\sigma$ be a strategy for \textbf{I} in $F[a,A]$ to reach $\Lambda_{a,p}^{-1}[\cal{C}]$. We define a strategy $\tau$ for \textbf{I} in $F[(a,p),(A,\vec{0})]$ as follows: We first let $\tau(\emptyset) := (\sigma(\emptyset),\vec{0})$, and $t_\emptyset := \emptyset$. Now suppose, for the induction hypothesis, that for all states $s$ of $F[(a,p),(A,\vec{0})]$ for \textbf{II} of rank $n$, there exists a state $t_s$ of $F[a,A]$ for \textbf{II} of rank $2n$ such that $\Lambda_{a,p}(a(t_s)) = a(s)$ and $\tau(s) = (\sigma(t_s),\vec{0})$. Let $s$ be a state for \textbf{II} of rank $n + 1$, and suppose that $\last(s) = (x_n,\varepsilon)$, where $\varepsilon \in \{0,1\}$. We let $t_s := {t_{s\restrictedto n}}^\frown(x_n,y_n)$, where $y_n$ is any element of $E$ such that $y_n \in Y$ iff $\varepsilon = 1$. Observe that $\Lambda_{a,p}(a(t_s)) = a(s)$. This finishes the construction of the strategy $\tau$. Now let $s$ be a complete play in $F[(a,p),(A,\vec{0})]$ following $\tau$. Since $\sigma$ is a strategy that reaches $\Lambda_{a,p}^{-1}[\cal{C}]$, $A(t_s) \in \Lambda_{a,p}^{-1}[\cal{C}]$, so:
    \begin{align*}
        A(s) = \Lambda_{a,p}(A(t_s)) \in \Lambda_{a,p}[\Lambda_{a,p}^{-1}[\cal{C}]] \subseteq \cal{C},
    \end{align*}
    as desired.
\end{proof}

\begin{lemma}
\label{lem:coding.preserves.G}
    Let $\cal{C} \subseteq E^{[\infty]} \times 2^\omega$. Let $A \in E^{[\infty]}$, $a \in E^{[<\infty]}\restrictedto A$ and $p \in 2^{\lh(a)}$. Suppose that \textbf{II} has a strategy in $G[a,A]$ to reach $\Lambda_{a,p}^{-1}[\cal{C}]$. Then \textbf{II} has a strategy in $G[(a,p),(A,\vec{0})]$ to reach $\cal{C}$.
\end{lemma}

\begin{proof}
    Let $\sigma$ be a strategy for \textbf{II} in $G[a,A]$ to reach $\Lambda_{a,p}^{-1}[\cal{C}]$. We define a strategy $\tau$ for \textbf{II} in $G[(a,p),(A,\vec{0})]$ as follows: For any $B \leq A$, we let $t_{(B)} := (B)$. Now let $s$ be a state of $G[(a,p),(A,\vec{0})]$ for \textbf{I}, and suppose that we have defined a corresponding state $t_s$ of $G[a,A]$ of rank $2n$ such that $\Lambda_{a,p}(a(t_s)) = a(s)$, and $\last(s) = (\last(t_s),\vec{0})$. Let $x_n,y_n$ be such that $x_n = \sigma(t_s)$, and $y_n = \sigma({t_s}^\frown(A))$. We then let:
    \begin{align*}
        \tau(s) := 
        \begin{cases}
            (x_n,1), &\text{if $y_n \in Y$}, \\
            (x_n,0), &\text{if $y_n \notin Y$}.
        \end{cases}
    \end{align*}
    Now let $t_{s^\frown{(B)}} := {t_s}^\frown(A,B)$. This finishes the construction of the strategy $\tau$, and by a similar reasoning to the last paragraph of Lemma \ref{lem:coding.preserves.F}, $\tau$ is a strategy for \textbf{II} in $G[(a,p),(A,\vec{0})]$ to reach $\cal{C}$.
\end{proof}

We thus obtain the following variant of Theorem \ref{thm:kr.and.projection}.

\begin{theorem}[Theorem IV.4.14, \cite{AT05}]
\label{thm:sr.and.projections}
    If every coanalytic subset of $E^{[\infty]}$ is strategically Ramsey, then every $\mathbf{\Sigma}_2^1$ subset of $E^{[\infty]}$ is strategically Ramsey. More generally, for every $n \geq 1$, if every $\mathbf{\Pi}_n^1$ subset of $E^{[\infty]}$ is strategically Ramsey, then every $\mathbf{\Sigma}_{n+1}^1$ subset of $E^{[\infty]}$ is strategically Ramsey.
\end{theorem}

\begin{proof}
    Suppose on the contrary that there exists a $\mathbf{\Sigma}_{n+1}^1$ non-strategically Ramsey subset of $E^{[\infty]}$. By Theorem \ref{thm:kr.and.projection}, there exists a $\mathbf{\Pi}_n^1$ subset $\cal{C} \subseteq E^{[\infty]} \times 2^\omega$ which is not strategically Ramsey. Therefore, there exists some $A \in E^{[\infty]}$ and $(a,p) \in (E^{[<\infty]} \times 2^{<\omega})\restrictedto(A,\vec{0})$ such that for all $(B,\vec{0}) \in [(a,p),(A,\vec{0})]$, neither \textbf{I} has a strategy in $F[(a,p),(B,\vec{0})]$ to reach $\cal{C}^c$, nor \textbf{II} has a strategy in $G[(a,p),(B,\vec{0})]$ to reach $\cal{C}$. By Lemma \ref{lem:coding.preserves.F} and Lemma \ref{lem:coding.preserves.G}, this implies that for all $B \in [a,A]$, neither \textbf{I} has a strategy in $F[a,B]$ to reach $\Lambda_{a,p}^{-1}[\cal{C}]^c = \Lambda_{a,p}^{-1}[\cal{C}^c]$, nor \textbf{II} has a strategy in $G[a,B]$ to reach $\Lambda_{a,p}^{-1}[\cal{C}]$. Therefore, $\Lambda_{a,p}^{-1}[\cal{C}]$ is a $\mathbf{\Pi}_n^1$ set (as $\Lambda_{a,p}$ is continuous) which is not strategically Ramsey.
\end{proof}

\begin{corollary}
    Suppose that there exists a $\Sigma_2^1$-good well-ordering of the reals. Then there exists a coanalytic subset of $E^{[\infty]}$ which is not strategically Ramsey.
\end{corollary}

\begin{proof}
    By Theorem \ref{thm:sigma_1^2.wo.gives.non.kr.set}, there exists a coanalytic $\cal{C} \subseteq E^{[\infty]} \times 2^\omega$ which is not strategically Ramsey. Now apply Theorem \ref{thm:sr.and.projections}.
\end{proof}

We remark that we have also essentially proved Corollary IV.4.13 of \cite{AT05}: If $\X \subseteq E^{[\infty]}$ is analytic, then $\X = \pi_0[\cal{C}]$ for some $G_\delta$ subset $\cal{C} \subseteq E^{[\infty]} \times 2^\omega$. Then $\Lambda_{a,p}^{-1}[\cal{C}]$ is also a $G_\delta$ subset of $E^{[\infty]}$ for all $a,p$, as $\Lambda_{a,p}$ is always continuous and $E^{[\infty]}(a)$ is a clopen subset of $E^{[\infty]}$. Therefore, proving that every $G_\delta$ subset of $E^{[\infty]}$ is strategically Ramsey would imply that every analytic subset of $E^{[\infty]}$ is strategically Ramsey.

\section{Further remarks and open questions}
Todor\v{c}evi\'{c} proved in \cite{T10} that if $(\cal{R},\leq,r)$ is a closed triple satisfying \textbf{A1}-\textbf{A4}, then every Suslin-measurable subset of $\cal{R}$ is Ramsey. This is strictly stronger than Corollary \ref{cor:analytically.measurable.is.ramsey}, which is our best conclusion from our general results regarding Kastanas Ramsey sets. However, by Proposition \ref{prop:kr.not.closed.under.intersections}, Kastanas Ramsey sets need not be closed under the Suslin operation in general.

\begin{qn}
    Can we prove a general result about Kastanas Ramsey subsets of a \textbf{wA2}-space which, when restricted to the setting of topological Ramsey space, directly implies that Ramsey subsets are closed under the Suslin operation?
\end{qn}

There are various set-theoretic properties shared by Ramsey subsets of topological Ramsey spaces and strategically Ramsey subsets of countable vector spaces, but it is not apparent to us if one can provide general results (in the context of \textbf{wA2}-spaces) which encompass them. One such property concerns the statement ``Every set is Kastanas Ramsey''. It is a classic result that in Solovay's model, every subset of a Polish space has the property of Baire. Since the Ellentuck topology refines the Polish topology, we have the following:

\begin{theorem}
    Let $(\cal{R},\leq,r)$ be a closed triple satisfying \textbf{A1}-\textbf{A4} such that $\cal{AR}$ is countable. Let $\kappa$ be an inaccessible cardinal, and let $\G$ be $\func{Coll}(\omega,<\kappa)$-generic. Then in the model $\L(\R)^{\V[\G]}$, every subset of $\cal{R}$ is Ramsey.
\end{theorem}

On the other hand, we have the following property of strategically Ramsey sets:

\begin{theorem}[Lopez-Abad, \cite{L05}]
    Let $\kappa$ be a supercompact cardinal, and let $\G$ be $\func{Coll}(\omega,<\kappa)$-generic. Then in the model $\L(\R)^{\V[\G]}$, every subset of $E^{[\infty]}$ is strategically Ramsey.
\end{theorem}

This leads to the following conjecture.

\begin{conjecture}
    Let $(\cal{R},\leq,r)$ be a \textbf{wA2}-space, and assume that $\cal{AR}$ is countable. It is consistent with sufficiently large cardinal assumptions that every subset of $\cal{R}$ is Kastanas Ramsey.
\end{conjecture}

We conclude with a question that naturally extends Corollary \ref{cor:r.times.2^omega.has.coanalytic.non.kr.set}:

\begin{qn}
\label{qn:analytic.non.kr.set}
    Let $(\cal{R},\leq,r)$ be a sufficiently well-behaved \textbf{wA2}-space. Assume that it has a biasymptotic set, and that $\cal{AR}$ is countable. If there exists a $\Sigma_2^1$-good well-ordering of the reals, then must there exist a coanalytic subset of $\cal{R}$ which is not Kastanas Ramsey?
\end{qn}

\section*{Acknowledgements}
The author would like to express his deepest gratitude to Spencer Unger and Christian Rosendal for their helpful comments, suggestions and corrections.

\printbibliography[heading=bibintoc,title={References}]

\end{document}

%% file: preamble.tex
\usepackage{amsmath, amsthm, amssymb, amsfonts}
\usepackage{bbm}
\usepackage{bm}
\usepackage{graphicx}
\usepackage{xcolor}
\usepackage{hyperref}
\usepackage{enumitem}
\usepackage{cleveref}

\theoremstyle{plain}

\usepackage[backend=biber,sorting=nty]{biblatex}

\numberwithin{equation}{section}

\newtheorem{theorem}{Theorem}[section]
\newtheorem{lemma}[theorem]{Lemma}
\newtheorem{proposition}[theorem]{Proposition}
\newtheorem{corollary}[theorem]{Corollary}
\newtheorem*{claim}{Claim}

\theoremstyle{definition}
\newtheorem{definition}[theorem]{Definition}
\newtheorem{notation}[theorem]{Notation}
\newtheorem{example}[theorem]{Example}

\newtheorem{qn}{Question}
\newtheorem*{conjecture}{Conjecture}
\newtheorem*{axiom}{Axiom}

\newcommand{\func}{\operatorname}

\renewcommand{\c}[1]{\left\langle #1 \right\rangle}

\newcommand{\R}{\mathbb{R}}

\newcommand{\Z}{\mathbb{Z}}

\newcommand{\Po}{\cal{P}}
\newcommand{\N}{\mathbb{N}}

\renewcommand{\bar}[1]{\overline{#1}}

\newcommand{\iffbreak}{\phantom{{} \iff {}}}

\newcommand{\cal}{\mathcal}

\renewcommand{\bf}{\mathbf}

\newcommand{\bcal}[1]{\bm{\mathcal{#1}}}

\newcommand{\supp}{\func{supp}}

\newcommand{\rank}{\func{rank}}

\newcommand{\dom}{\func{dom}}
\newcommand{\ran}{\func{ran}}

\newcommand{\restrictedto}{\mathord{\upharpoonright}}

\newcommand{\FIN}{\bf{FIN}}

\newcommand{\V}{\bf{V}}

\newcommand{\X}{\cal{X}}

\renewcommand{\H}{\cal{H}}

\newcommand{\lh}{\func{lh}}
\newcommand{\G}{\cal{G}}
\newcommand{\last}{\func{last}}

\renewcommand{\O}{\cal{O}}
\newcommand{\depth}{\func{depth}}

\newcommand{\fin}{\mathrm{fin}}
\newcommand{\D}{\cal{D}}

\newcommand{\E}{\cal{E}}

\newcommand{\height}{\func{height}}

\newcommand{\blhd}{\blacktriangleleft}

\let\strokeL\L
\DeclareRobustCommand{\L}{\ifmmode\mathbf{L}\else\strokeL\fi}

\let\Sec\S
\renewcommand{\S}{\mathcal{S}}

\makeatletter
\DeclareFontFamily{OMX}{MnSymbolE}{}
\DeclareSymbolFont{MnLargeSymbols}{OMX}{MnSymbolE}{m}{n}
\SetSymbolFont{MnLargeSymbols}{bold}{OMX}{MnSymbolE}{b}{n}
\DeclareFontShape{OMX}{MnSymbolE}{m}{n}{
    <-6>  MnSymbolE5
   <6-7>  MnSymbolE6
   <7-8>  MnSymbolE7
   <8-9>  MnSymbolE8
   <9-10> MnSymbolE9
  <10-12> MnSymbolE10
  <12->   MnSymbolE12
}{}
\DeclareFontShape{OMX}{MnSymbolE}{b}{n}{
    <-6>  MnSymbolE-Bold5
   <6-7>  MnSymbolE-Bold6
   <7-8>  MnSymbolE-Bold7
   <8-9>  MnSymbolE-Bold8
   <9-10> MnSymbolE-Bold9
  <10-12> MnSymbolE-Bold10
  <12->   MnSymbolE-Bold12
}{}

\let\llangle\@undefined
\let\rrangle\@undefined
\DeclareMathDelimiter{\llangle}{\mathopen}%
                     {MnLargeSymbols}{'164}{MnLargeSymbols}{'164}
\DeclareMathDelimiter{\rrangle}{\mathclose}%
                     {MnLargeSymbols}{'171}{MnLargeSymbols}{'171}
\makeatother

\newenvironment{midproof}[1][Proof]
  {\begin{proof}[#1]}
  {\end{proof}}

\DeclareFontShape{OT1}{cmr}{bx}{sc}{<-> cmbcsc10}{}